\documentclass{article}
\usepackage[left=3cm,right=3cm,top=2.5cm,bottom=2.5cm,includeheadfoot]{geometry}
\usepackage{etoolbox}
\usepackage{graphicx}
\usepackage{amsmath,amssymb, amsthm}
\usepackage{dsfont}
\usepackage{enumitem}

\usepackage{mathtools}
\usepackage{multirow}
\usepackage{algorithm2e}
\usepackage{bm}
\usepackage{hyperref}

\newtheorem{definition}{Definition}
\newtheorem{remark}{Remark}
\newtheorem{lemma}{Lemma}
\newtheorem{proposition}{Proposition}
\newtheorem{theorem}{Theorem}

\usepackage{amsmath}
\usepackage{bm}
\usepackage{graphicx}
\usepackage{amssymb}
\usepackage{graphicx}
\usepackage{color}
\usepackage{tikz}
\usepackage{pgfplots}
\usepackage{algorithm2e}
\usepackage{tabularx}
\usepackage{url}
\usepackage{xparse}
\usepackage{multirow}
\usepackage{dsfont}
\pgfplotsset{compat=1.13}
\usepackage{mathtools}

\newcommand{\Id}{\mathds{1}}
\newcommand{\warp}{\mathbf{T}}
\newcommand{\R}{\mathbb{R}}
\newcommand{\N}{\mathbb{N}}
\DeclareMathOperator*{\argmin}{\operatorname{argmin}}
\newcommand{\tr}{\operatorname{}{tr}}
\renewcommand{\div}{\operatorname{}{div}}
\newcommand{\sym}{\text{sym}}
\renewcommand{\d}{\,\mathrm{d}}
\newcommand{\bigO}{\mathcal{O}}

\newcommand{\image}{u}
\newcommand{\deformation}{\phi}
\newcommand{\imagevec}{\mathbf{u}}
\newcommand{\defvec}{\mathbf{\Phi}}
\newcommand{\accvec}{\mathbf{a}}

\newcommand{\imageSpace}{\mathcal{I}}
\newcommand{\deformationSpace}{\mathcal{D}}
\newcommand{\motionSpace}{\mathcal{V}}
\newcommand{\admset}{\mathcal{D}}

\newcommand{\domain}{\Omega}

\newcommand{\energyDensity}{W}
\newcommand{\Pathenergy}{\mathbf{E}}
\newcommand{\SplinePathenergy}{\mathbf{F}}

\newcommand{\pathenergy}{\mathcal{E}}
\newcommand{\splinepathenergy}{\mathcal{F}}

\newcommand{\discreteDomain}{\Omega_{\scriptscriptstyle{MN}}}
\newcommand{\MN}{\scriptscriptstyle{MN}}
\newcommand{\discreteImage}{\mathbf{u}}
\newcommand{\discreteDerivative}{\mathbf{z}}
\newcommand{\discreteDeformation}{\boldsymbol{\phi}}
\newcommand{\discreteAcceleration}{\mathbf{a}}
\newcommand{\DataEnergy}{\mathbf{D}}
\newcommand{\discretex}{\mathbf{x}}
\newcommand{\discretey}{\mathbf{y}}

\definecolor{blue}{rgb}{0.137255,0,0.862745}
\definecolor{red}{rgb}{0.862745,0.137255,0}
\definecolor{myOrange}{RGB}{255, 169, 87}
\definecolor{myGreen}{RGB}{180, 255, 162}
\definecolor{myGrey}{RGB}{187, 187, 187}
\definecolor{myDarkGrey}{RGB}{100, ,100, 100}

\newcommand{\change}[1]{{#1}}

\begin{document}

\title{Consistent Approximation of Interpolating Splines in Image Metamorphosis
}

\author{Jorge Justiniano\thanks{Institute for Numerical Simulation, University of Bonn (\url{name.surname@ins.uni-bonn.de})}  \and
        Marko Rajkovi\'c\footnotemark[1]  \and
        Martin Rumpf\footnotemark[1]
}

\maketitle

\begin{abstract}
This paper investigates a variational model for splines in the image metamorphosis model
for \change{the smooth} interpolation of key frames in the space of images. 
\change{ The Riemannian manifold of images based on the metamorphosis model defines shortest  
geodesic paths interpolating two images as minimizers of the path energy which measures the viscous dissipation caused by the motion field and dissipation caused by the material derivative of the image intensity along motion paths.}
\change{In this paper we aim at smooth interpolation of multiple key frame images 
picking up the general observation of 
cubic splines in Euclidean space which minimize the squared acceleration along the interpolation path.
To this end, we propose the spline functional which combines quadratic functionals of the 
Eulerian motion acceleration and of the 
second material derivative of the image intensity as the proper notion of image intensity acceleration.}
\change{We propose a variational time discretization of this  model and study the convergence to a suitably relaxed time continuous model via $\Gamma$-convergence methodology.}
As a byproduct, this also allows to establish the existence of metamorphosis splines for given key frame images as minimizers of the \change{time} continuous  spline functional.
The time discretization is complemented by effective spatial discretization based on finite differences and 
a stable B-spline interpolation of deformed quantities.
A variety of numerical examples demonstrates the robustness and versatility of the proposed method in applications. For the minimization of the fully discrete energy functional a variant of the iPALM algorithm is used\footnote{
This publication is an extended version of the previous conference proceeding~\cite{JuRaRu21} presented at SSVM 2021.}.
%\keywords{Image metamorphosis \and Image morphing \and Spline interpolation \and  $\Gamma$-convergence \and iPALM algorithm}
\end{abstract}

\section{Introduction}
\label{sec:intro}
\change{At first, we briefly review image metamorphosis as a flexible model for image morphing which generalizes the flow of diffeomorphism approach (cf. the textbook by Younes \cite{Yo10}). It has been extensively investigated by Trouv\'e, Younes and coworkers~\cite{TrYo05,TrYo05a}. Image metamorphosis is a Riemannian manifold approach to the space of 
images and can be used to interpolate two images by a connecting geodesic path. Thereby, geodesic paths are 
minimizers  of the associated path energy  
over all regular paths connecting a pair of input images.  
The underlying metric associates a cost both to the transport of image intensities via a viscous flow and to image intensity variations along motion paths. 
The path energy is given as time integral over the metric evaluated on 
a pair of transport velocities and material derivatives of image intensities.}  The metamorphosis model has been extended to discrete measures~\cite{RiYo13}, to reproducing kernel Hilbert spaces~\cite{RiYo16}, to functional shapes~\cite{ChChTr16}, to images on Hadamard manifolds~\cite{EfNe19} and to deep feature spaces~\cite{EfKoPo19a}.
\change{Based on the time discrete metamorphosis model proposed in \cite{BeEf14},
Effland \textit{et al.} introduced the B\'{e}zier curves interpolation in the space of images \cite{EfRuSi14}, image geodesics for optical coherence tomography \cite{BeBuEf17} and discrete extrapolation \cite{EfRuSc17a}.
The (time discrete) metamorphosis model was further used for regularization of several inverse problems in imaging:
exemplar-based face colorization \cite{PePiSt17},
sparse and limited angle computerized tomography and super-resolution of images \cite{NePeSt19}, and joint (spatio-temporal) tomographic reconstruction and registration \cite{GrChOk20}.}

\change{Geodesic paths are the generalization of straight lines in Euclidean space to Riemannian manifolds.
The need for an as smooth as possible interpolation of multiple data points is an extensively studied problem in mathematics. Cubic splines are a widespread and versatile solution 
for this problem.
Due to the famous result by de Boor~\cite{dB63}, in Euclidean space  this approach amounts 
to finding a path $t\mapsto u(t)$ which minimizes the integral of the squared acceleration $\int_0^1 |\ddot u(t)|^2 \d t$.
 %the minimizers of this energy are cubic splines.
In a Riemannian context, Noakes et al.~\cite{NoHePa89} introduced Riemannian cubic splines as stationary paths of the integrated squared covariant derivative of the velocity.
In this paper, we aim at generalizing this ansatz to the Riemannian space of images equipped 
with the metamorphosis metric.
%The data points to be interpolated are in this cases images.
Piecewise geodesic interpolation of images lacks smoothness at the interpolation data points leading to a jerking 
at the corresponding times when animating the resulting family of images. 
To overcome this shortcoming, we consider as the spline energy the sum of
 two integrals over squared acceleration quantities.
The first quantity is the classical transport based Eulerian acceleration and the second  
quantity represents the second order material derivative of the image intensity.
Minimizers of this spline energy are one way to naturally generalize Euclidean splines to the metamorphosis model on the space of images. In fact, our model separates in a physically intuitive way the Eulerian flow acceleration, and the second material derivative of the image intensity.}
This way differs from the spline energy introduced by Noakes
with the spline energy being the integral over the squared covariant derivative of the path velocity 
(cf. also \cite{NoHePa89,TaVi19,Vi20} or \cite{HeRuWi18}). In this case the squared covariant derivative of the path velocity with respect to the Riemannian metric would lead to an interwoven model of the different types of acceleration.

% In this paper, we discuss a spline energy functional as a second order extension of the first order path energy. Given a set of key frames at 
% disjoint times a spline path is given as a minimizer of the spline energy subject to the key frame interpolation constraint.
%In Euclidean space cubic splines $t\mapsto u(t)$ are known to be minimizers of the integral of the squared acceleration 
%$\int_0^1 |\ddot u(t)|^2 \d t$ due  to the famous result by de Boor~\cite{dB63}. In a Riemannian context, Noakes et al.~\cite{NoHePa89} introduced Riemannian cubic splines as stationary paths of the integrated squared covariant derivative of the velocity.
\change{One of the first applications of splines in computer vision was work by Mumford~\cite{Mu94} where splines appear as a maximum likelihood reconstruction of occluded edges due to the penalization of curvature of arc length parametrized curves.}
Today, there is a variety of spline approaches in nonlinear spaces and with applications to shape spaces.
Trouv\'e and Vialard~\cite{TrVi12} studied a second-order shape functional in landmark space based on an optimal control approach.
Singh et al.~\cite{SiNiVi15} introduced an optimal control method involving a functional which measures the motion acceleration in a flow of  diffeomorphisms ansatz for image regression.
Tahraoui and Vialard~\cite{TaVi19} consider a second-order variational model on the group of diffeomorphisms of the interval $[0,1]$. 
%involving the Eulerian acceleration in the context to diffeomorphic flow.
They proposed a relaxed model leading to a Fisher-Rao functional, as a convex functional on the space of measures. Vialard \cite{Vi20} showed the existence of a minimizer of the Riemannian acceleration energy on the group of diffeomorphisms endowed with a right-invariant Sobolev metric of high order. 
Benamou et al.~\cite{BeGaVi19} and Chen et al.~\cite{ChCoGe18}
discussed spline interpolation in the space of probability measures endowed with the Wasserstein metric. Thereby, energy splines are defined as minimizers of the action functional on Wasserstein space which involves the acceleration of measure-valued paths, sharing similarities with the spline functional in the space of images introduced in this paper.
The initial computational intractability of such approaches is tackled by a  relaxation based on multi-marginal optimal transport and entropic regularization.
The transport problem this approach aims at solving might not have a Monge solution,
which was remedied by a new method introduced by Chewi et al.~\cite{ChClGo+20} to construct measure-valued splines, dubbed transport splines. This method additionally enjoys substantial computational advantages.

\change{In this paper,} we will recall the image metamorphosis model and the corresponding path energy whose minimizers are geodesic paths. \change{On that basis we then derive the novel spline energy functional. 
Given a set of key frames at 
disjoint times a spline curve is then given as a smooth in time minimizer of this spline energy respecting 
the key frame images as interpolation constraints.}
Furthermore, we will study a suitable time discrete variational model, which generalizes the time discrete 
metamorphosis model proposed in~\cite{BeEf14,EfKoPo19a} . The central contribution of this paper is the convergence of this time discrete model 
to the time continuous metamorphosis spline model in the sense of Mosco~\cite{Mo69}. As a consequence, one obtains existence of metamorphosis spline paths.
Finally, we discretize the model in space and we derive a numerical scheme to solve for fully discrete metamorphosis spline paths in the space of images.

\paragraph{Notation.}
Throughout this paper, we assume that the image domain~$\domain\subset\R^d$ for $d\in\{2,3\}$ is bounded and strongly Lipschitz.
We use standard notation for Lebes\-gue and Sobolev spaces from the image domain~$\domain$ to a Banach space~$X$,
i.e.~$L^p(\domain,X)$ and $H^m(\domain,X)$ and omit $X$ if the space is clear from the context.
The associated $X$ norms are denoted by $\Vert\cdot\Vert_{L^p(\domain)}$ and $\Vert\cdot\Vert_{H^m(\domain)}$, respectively, and the seminorm in $H^m(\domain)$ is given by $|\cdot|_{H^m(\domain)}$, i.e.
\[
|f|_{H^m(\domain)}=\Vert D^m f\Vert_{L^2(\domain)}\,,\quad
\Vert f\Vert_{H^m(\domain)}^2=\sum_{j=0}^m|f|_{H^j(\domain)}^2
\]
for $f\in H^m(\domain)$.
We use the notation $C^{k,\alpha}(\overline \domain,X)$ for H\"older spaces of order $k\geq0$ with H\"older regularity $\alpha\in(0,1]$ for the $k$-th derivatives and the corresponding (semi)norm is
\begin{align*}
|f|_{C^{0,\alpha}(\overline\domain)}=\sup_{x\neq y\in\domain}\frac{|f(x)-f(y)|}{|x-y|^\alpha}\,, \quad
\Vert f\Vert_{C^{k,\alpha}(\overline\domain)}=\Vert f\Vert_{C^k(\overline\domain)}+\sum_{|\beta|=k}|D^{\beta}f|_{C^{0,\alpha}(\overline \domain)}\,.
\end{align*}
The space $AC^p([0,1],X)$ is the space of absolutely continuous functions with the derivative in $L^p((0,1))$. 
The symmetric part of a matrix $A\in \R^{d,d}$ is denoted by $A^\sym$, i.e.~$A^\sym=\frac{1}{2}(A\!+\!A^\top)$ and the symmetrized Jacobian of a differentiable function $\phi$ by $\varepsilon[\phi]\!=\!(D\phi)^{\sym}$.
We denote  by $\Id$ both the identity map and the identity matrix.
Finally, $f_t$ refers to the evaluation of the function at time $t$, while $\dot{f}$ refers to the temporal derivative of a differentiable function~$f$.

\paragraph{Organization. }
This paper is organized as follows. In Section~\ref{sec:review} we review the most important properties of the flow of diffeomorphism model and the metamorphosis model as its extension.
In Section~\ref{sec:time_continuous} the time continuous spline energy is derived  
and the proper interplay between Lagrangian and Eulerian perspective is discussed.
Then, in Section~\ref{section:time_discrete} a variational time discretization of the continuous spline energy is introduced and the existence of discrete splines is studied. Section~\ref{sec:time_extension} represents the time extension of the time discrete quantities, which allows study of convergence to the time continuous model, presented in Section~\ref{sec:convergence}.
Section~\ref{sec:fully_discrete} explains the fully discrete scheme and Section~\ref{sec:algorithm}
shows how to set up a suitable iPALM algorithm to numerically solve for a spline interpolation 
given a set of key frames.
Finally, Section~\ref{sec:results} experimentally demonstrates properties of the spline approach for image metamorphosis and shows applications of the proposed method.

In comparison with the conference proceeding \cite{JuRaRu21} we present a significantly more detailed study of both time discrete and time continuous metamorphosis splines. We prove existence of discrete metamorphosis splines, give a suitable continuous extension of discrete metamorphosis splines and show that the discrete spline functional converges in the sense of Mosco to the continuous spline functional. This in particular enables to establish the existence of continuous metamorphosis splines as minimizers of the continuous spline energy.  Furthermore, we extended the applications section.

\section{Review of Metamorphosis model}\label{sec:review}
In this section, we briefly review the classical flow of diffeomorphism model and the metamorphosis model as its generalization.

\subsection{Flow of Diffeomorphism}
\change{The flow of diffeomorphism model~\cite{DuGrMi98,BeMiTr05,JoMi00,MiTrYo02} is based on Arnold's paradigm \cite{Ar66a} to study of flow of ideal fluids by geodesics on the group of
volume-preserving diffeomorphisms. To this end, given a domain $\domain \subset \R^d$, one considers a family of diffeomorphisms $\left(\psi_t\right)_{t\in [0,1]}:\overline\domain\to\R^d$ determined by its time-dependent {Eulerian velocity} $v_t=\dot\psi_t\circ\psi^{-1}_t$. The Riemannian structure on this space is given by the metric
\begin{align*}
g_{\psi_t}(\dot{\psi}_t,\dot{\psi}_t)\coloneqq\int_\domain L[v,v]\d x,
\end{align*}
and the path energy
\begin{equation*}
\pathenergy_{\psi_t}[(\psi_t)_{t\in [0,1]}]\coloneqq\int^1_0 g_{\psi_t}(\dot{\psi}_t,\dot{\psi}_t)\d t.
\end{equation*}
}
%In the flow of diffeomorphism model, the temporal chan\-ge of $c$-channel image intensities $(\image_t)_{t \in [0,1]}:\domain \to \R^c$, is determined by a {family of diffeomorphisms} $(\psi_t)_{t\in [0,1]}:\overline\domain\to\R^d$ 
%describing a flow transporting image intensities along particle paths. This transport is given in terms of the equation $\image_t(\cdot)=\image_0\circ\psi^{-1}_t(\cdot)$ also known as \emph{brightness constancy assumption} \cite{HoSc81}, which
%is equivalent to a vanishing material derivative~$\frac{D}{\partial t}\image=\dot\image+v\cdot D\image$,
%where $v_t=\dot\psi_t\circ\psi^{-1}_t$ denotes the time-dependent {Eulerian velocity}.
%The Riemannian space of images is endowed with the path energy
%\begin{equation*}
%\pathenergy_{\psi_t}[(\psi_t)_{t\in [0,1]}]=\int^1_0 g_{\psi_t}(\dot{\psi}_t,\dot{\psi}_t)\d t,
%\end{equation*}
%where the metric is given by
%\begin{align*}
%g_{\psi_t}(\dot{\psi}_t,\dot{\psi}_t)=\int_\domain L[v,v]\d x.
%\end{align*}
Here, $L$ defines a quadratic form corresponding to a higher order elliptic operator. The specific choice, used throughout this paper is
\begin{equation}
L[v,v]\change{\coloneqq} \tr (\varepsilon[v]^2) + \gamma |D^m v|^2,~m>1+\frac{d}{2},\gamma >0. \label{eq:L_definition}
\end{equation}
In this case the metric $g_{\psi_t}(\dot \psi_t,\dot \psi_t)$ describes the viscous dissipation in a multipolar fluid model as investigated by Ne\v{c}as and \v{S}ilhav\'y~\cite{NeSi91}.
The first term of the integrand represents the dissipation density in a simple Newtonian fluid and the second term can be regarded as a higher order measure of the fluid friction.
\change{Using that the metric $g_{\psi_t}$ is $H^m(\domain)$-coercive \cite[Theorem 3.1]{DuGrMi98} shows the existence of a flow of diffeomorphisms as a minimizer of the above energy. This minimizer 
represents a geodesic path connecting two fixed diffeomorphisms.}

\change{In the context of image morphing the flow of diffeomorphism is transporting image intensities along particle paths describing the temporal chan\-ge of $c$-channel image intensities $(\image_t)_{t \in [0,1]}:\domain \to \R^c$. This transport is given in terms of the equation $\image_t(\cdot)\coloneqq\image_0\circ\psi^{-1}_t(\cdot)$ also known as \emph{brightness constancy assumption} \cite{HoSc81}, which
is equivalent to a vanishing material derivative~$\frac{D}{\partial t}\image\coloneqq\dot\image+v\cdot D\image$.}
Given two image intensity functions $\image_A,\image_B$, an associated geodesic path is a family of images subject to the \change{constraint}
$\image_0=\image_A$, $\image_1=\image_B$ and $\image_t(\cdot)=\image_A\circ\psi^{-1}_t(\cdot)$  where the underlying family of diffeomorphisms $\left(\psi_t\right)_{t\in [0,1]}$ minimizes the path energy. 

%We refer the reader to~\cite{DuGrMi98,BeMiTr05,JoMi00,MiTrYo02} for further details.

\subsection{Metamorphosis Model}
The metamorphosis approach originally proposed by Miller, Trouv\'e, Younes and coworkers in~\cite{MiYo01,TrYo05,TrYo05a}
generalizes the flow of diffeomorphism model by allowing for image intensity variations along motion paths 
and penalizing the squared material derivative in the metric. 
Under the assumption that the image path~$\image$ is sufficiently smooth, the metric and the path energy read as
\begin{align}
g_\image(\dot\image,\dot\image)&\change{\coloneqq}\min_{v:\overline\domain\to\R^d}\int_\domain L[v,v]+\frac1\delta \left|\frac{D}{\partial t}\image\right|^2\d x,\nonumber\\
\pathenergy[\image]&\change{\coloneqq}\int_0^1 g_{\image_t}(\dot\image_t,\dot \image_t)\d t.\label{eq:metamorphosis_initial_energy}
\end{align}
Hence, the flow of diffeomorphism model is the limit case of the metamorphosis model for~$\delta\to 0$.

Since paths in the space of images do not exhibit any time nor space smoothness properties in general, the evaluation of the material derivative is not well-defined.
As a solution to these problems, Trouv\'e and Younes~\cite{TrYo05a} proposed
a nonlinear geometric structure in the space of images~$\imageSpace\coloneqq L^2(\domain,\R^c)$.
In detail, for a given image path~$\image\in L^2([0,1],\imageSpace)$ and an associated
velocity field~$v\in L^2((0,1),\motionSpace)$, where $\motionSpace \coloneqq H^{m}(\domain,\R^d)\cap H^{1}_0(\domain,\R^d)$, we define 
the vector valued \emph{weak material derivative} $\hat{z}\in L^2((0,1),L^2(\domain,\R^c))$ by
\begin{align*}
\int_0^1\int_\domain\eta \hat{z}\d x\d t=-\int_0^1\int_\domain(\partial_t\eta+\div(v\eta))\image\d x\d t,
\end{align*}
for a smooth test function $\eta\in C^{\infty}_c((0,1)\times\domain)$. 
With further consideration of equivalence classes of motion fields and material derivatives inducing the same temporal change of the image intensity we can consider $(v,\hat{z})$ as a tangent vector in the tangent space of $\imageSpace$ at the image $\image$ and write $(v,\hat{z}) \in T_\image\imageSpace$.
This gives rise to the notion $H^1([0,1],\imageSpace)$ for regular paths in the space of images.
The \emph{path energy in the metamorphosis model}
for a regular path $\image\in H^1([0,1],\imageSpace)$ is then defined as
\begin{equation}
\pathenergy[\image]\change{\coloneqq}\int_0^1\inf_{(v,\hat{z})  \in T_u\imageSpace} \int_\domain L[v,v]+\frac{1}{\delta}|\hat{z}|^2\d x\d t\,. \label{eq:DefinitionPathenergyImage}
\end{equation}
Image morphing of two input images~$\image_A,\image_B\in\imageSpace$ then amounts to computing a shortest geodesic path~$\image\in H^1([0,1],\imageSpace)$ in the metamorphosis model,
which is defined as a minimizer of the path energy in the class of regular curves such that $\image(0)=\image_A$ and $\image(1)=\image_B$.
The infimum in~\eqref{eq:DefinitionPathenergyImage} is attained, which is shown in~\cite[Proposition~1 \& Theorem~2]{TrYo05a}\change{, and} 
the existence of a shortest geodesic is proven in~\cite[Theorem 6]{TrYo05a}.

The Lagrangian formulation of variation of the image intensity along motion trajectories is 
\begin{align}\label{eq:first_variational_equality}
\image_t \circ \psi_t-\image_s \circ \psi_s=\int_s^t \hat{z}_r \circ \psi_r\d r, ~ \forall s,t\in[0,1]\,.
\end{align}
This motivates a relaxed approach proposed in \cite{EfNe19,EfKoPo19a}, 
where the material derivative quantity is retrieved from a variational inequality
\begin{align}\label{eq:variational_inequality}
|\image(t,\psi_t(x))-\image(s,\psi_s(x))|\leq\int_s^t z(r,\psi_r(x))\d r
\end{align}
for a.e.~$x\in \Omega$ and all $1\geq t > s \geq 0$. Formally, the scalar valued $z=|\hat z|$ replaces the actually vector-valued material derivative.
In fact, given a path $\image$ in $L^2([0,1],\imageSpace)$ the inequality \eqref{eq:variational_inequality} defines a set $\mathcal{C}(\image) \subset L^2((0,1),\motionSpace)\times L^2((0,1) \times \domain)$ of admissible pairs  $(v,z) \in \mathcal{C}(\image)$ of the velocity and scalar weak material derivative fulfilling this \change{inequality.
Then} the geodesic energy for a path $\image\in L^2([0,1],\imageSpace)$ is defined by
\begin{equation}\label{eq:geodesic_energy_scalar}
\pathenergy[\image]\change{\coloneqq}\int_0^1\inf_{(v,z)\in\mathcal{C}(\image)}\int_\domain L[v,v]+\frac{1}{\delta}z^2\d x\d t.
\end{equation}
The equivalence of the approaches \eqref{eq:DefinitionPathenergyImage} and \eqref{eq:geodesic_energy_scalar} follows from \cite[Proposition 8]{EfNe19}, in the sense that for every $z$ satisfying \eqref{eq:variational_inequality} there exists a $\hat{z}$ satisfying \eqref{eq:first_variational_equality} with $|\hat{z}|\leq z$ and for every $\hat{z}$ satisfying \eqref{eq:first_variational_equality} we have that $z=|\hat{z}|$ satisfies \eqref{eq:variational_inequality}. Furthermore, for every $u \in L^2([0,1],\imageSpace)$ with non-empty $\mathcal{C}[u]$ one can show that $u \in C^0([0,1],\imageSpace)$ (cf.~\cite[Remark~1]{EfKoPo19a}), which allows point evaluation in time.
In both \cite{EfNe19,EfKoPo19a} the existence of continuous time geodesic paths was shown, where the relaxed approach turned out to be very natural.

\section{Time Continuous Splines in Metamorphosis Model}\label{sec:time_continuous}
%In this section we introduce the time continuous splines in metamorphosis model.
\change{In this section we study the spline interpolation for a set of given key frame images in the time continuous setting of metamorphosis model.
In mathematical notation, given a set of $J$ key frames $u_j^I \in \imageSpace$ we ask for spline interpolation $(u_t)_{t\in[0,1]}$ which satisfies the constraints}
%We ask for spline interpolation $(u_t)_{t\in[0,1]}$ given a set of $J$ key frames $u_j^I \in \imageSpace$ with constraints
\begin{equation}\label{eq:constraint}
u_{t_j} = u_j^I,~t_j\in [0,1],~ j=1,\ldots J.
\end{equation}
To this end, we recall that  cubic splines in Euclidean space minimize the integral over the squared motion acceleration subject to position constraints \cite{dB63}, 
whereas linear interpolation is associated with the minimization of the integral over the squared motion velocity.  
In our case, image morphing via minimization of the path energy \eqref{eq:metamorphosis_initial_energy} corresponds to this linear interpolation. Here, we introduce the acceleration quantities which will be penalized in the spline energy.
The Eulerian flow acceleration is defined by
\begin{align}\label{eq:formal_flow_acceleration}
a_t\circ \psi_t\change{\coloneqq}\ddot{\psi}_t 
=\frac{\d}{\d t}(v_t \circ \psi_t)=\dot{v}_t \circ \psi_t + Dv_t \circ\psi_t\cdot v_t \circ \psi_t,
\end{align}
and for image paths with enough smoothness the second order material derivative is given by 
\begin{align}
&\frac{D^2}{\change{\partial}t^2}u\label{eq:Eulerian_w}
\change{\coloneqq}\frac{\d}{\d t} (\dot{u}_t + v_t \cdot Du_t) 
=\ddot{u}_t\!+\!v_t\!\cdot\! D\dot{u}_t\!+\!Du_t \cdot \left(\dot{v}_t \!+\! 
Dv_t v_t\right) \!+\!v_t \!\cdot\! \left(D\dot{u}_t \!+\! D^2u_t v_t\right). 
\end{align}
As already mentioned in the introduction this splitting of acceleration term into flow acceleration and second order change of image intensity does not fully correspond to a Riemannian manifold approach since the stronger interference of the two entities is observed in the covariant derivative of the vector field $(v,\frac{D}{\change{\partial}t}u)$. 
The spline energy comprises of the integral over the two quadratic acceleration quantities:
\begin{equation*}
\mathcal{F}[u]\coloneqq\min_{v}\int_0^1 \int_{\domain} {L}[a,a] + \frac{1}{\delta} \left|\frac{D^2}{\change{\partial}t^2} \image\right|^2 \d x \d t\,,
\end{equation*}
where for the simplicity we use the same elliptic operator \eqref{eq:L_definition}.
%\todo[inline]{it turns out later that we have the same regularity for $a$ as for $v$. Thus, we could use $L^m$ with $m>1+\frac{d}{2}$ and there would, actually, be no need for $m$ in notation of $L$! I suggest the following sentence to introduce this (with $m-1$ being replaced by $m$)}

As in the case of geodesic path energy, we give rigorous formulations for general paths $u \in L^2((0,1),\imageSpace)$.
We 
%denote the space 
%$\mathcal{A}=H^1_0(\domain,\R^d) \cap H^{m}(\domain,\R^d)$ and
observe the pairs $(v,a) \in L^2((0,1), \motionSpace) \times  L^2((0,1),\motionSpace)$ which determine the system
\begin{align*}
v_t(\psi_t(x))&=\dot{\psi}_t(x), ~ \psi_0(x)=x \\
a_t(\psi_t(x))&=\ddot{\psi}_t(x)\change{,}
\end{align*}
\change{and the corresponding  diffeomorphic flow $\psi \in H^2((0,1), \\ H^m(\domain,\domain))$.}
Notice that from \eqref{eq:formal_flow_acceleration}, one expects that the acceleration has one derivative less in comparison to the velo\-city. However, the approach we later take allows us to have the same number of derivatives. A similar gain of smoothness was noticed in \cite{Vi20}.

Similar to \eqref{eq:first_variational_equality} the Lagrangian formulation of  the material acceleration $\hat w \in L^2((0,1),L^2(\domain,\R^c))$ is given by
\begin{equation*}
\int_s^t \hat w_r \circ \psi_r \d r = \hat z_t \circ \psi_t-\hat z_s \circ \psi_s\,.
\end{equation*}
By further using \eqref{eq:first_variational_equality} we have
\begin{align}
\int_0^\tau \int_s^t \hat{w}_{r+l} \circ \psi_{r+l} \d r \d l 
=&\int_0^\tau \hat z_{t+l} \circ \psi_{t+l} - \hat z_{s+l} \circ \psi_{t+l} \d l \nonumber\\
=& (u_{t+\tau} \circ \psi_{t+\tau}\!-\!u_t \circ \psi_t)\!-\!(u_{s+\tau} \circ \psi_{s+\tau}\!-\!u_s \circ \psi_s)\,, \label{eq:second_derivative_central_diff_eq}
\end{align} for all $s,t \in (0,1)$ and every $\tau$, such that $t+\tau, s+\tau\in[0,1]$. Observe that for $s=t-\tau$ the \change{right hand side} of \eqref{eq:second_derivative_central_diff_eq} is an integral version of the second order central difference.
Because there is no differentiation involved in these definitions, they work for general image paths.
Analogous to \eqref{eq:variational_inequality}, we introduce the scalar quantity $w \in L^2((0,1)\times \domain)$ as an relaxation of the weak second order material derivative
\begin{align}
%\label{eq:variational_inequality} \\
\int_0^\tau \int_s^t  w_{r+l} \circ \psi_{r+l} \d r \d l 
\geq  \left| u_{t+\tau} \circ \psi_{t+\tau}-u_t \circ \psi_t-u_{s+\tau} \circ \psi_{s+\tau}+u_s \circ \psi_s \right|\change{,} \label{eq:second_central_variational_inequality} 
\end{align}
\change{for every $s \leq t \in (0,1), \tau \geq 0, t+\tau \leq 1$.}
 This relaxed Lagrangian approach is substantially more handsome in \change{comparison} to the Eulerian approach \eqref{eq:Eulerian_w} which will be exploited in the proof of consistency of continuous and time discrete approaches. 

We can show the equivalence of the two approaches corresponding to \eqref{eq:second_derivative_central_diff_eq} and \eqref{eq:second_central_variational_inequality} \change{(cf.~\cite[Proposition 8]{EfNe19})}.
\begin{proposition}\label{prop:equivalence_of_approaches}
For every vector valued $(\hat{z},\hat{w})$ satisfying \eqref{eq:first_variational_equality} and \eqref{eq:second_derivative_central_diff_eq} there exist scalar quantities $(z,w)$ satisfying \eqref{eq:variational_inequality} and \eqref{eq:second_central_variational_inequality} with $z=|\hat{z}|$ and $w=|\hat{w}|$. Conversely, for every $(z,w)$ satisfying \eqref{eq:variational_inequality} and \eqref{eq:second_central_variational_inequality} there exists $(\hat{z},\hat{w})$ satisfying \eqref{eq:first_variational_equality} and \eqref{eq:second_derivative_central_diff_eq} with $ z\geq|\hat{z}|$ and $w\geq|\hat{w}|$.
\end{proposition}
\begin{proof}
The first claim is obvious by the triangle inequality. To prove the converse, let $z$ satisfy \eqref{eq:variational_inequality}. We have
\begin{equation*}
\|u_t \circ \psi_t-u_s \circ \psi_s\|_{L^2(\domain)}\leq \int_s^t \| z_r\circ\psi_r\|_{L^2(\domain)} \d r,
\end{equation*}
from where we conclude that the function $t \mapsto \image_t \circ \psi_t$ is in $AC^2([0,1],L^2(\domain,\change{\R^c}))$ \change{which implies} the a.e. differentiability and the existence of derivative ${z}' \in L^2((0,1),L^2(\domain,\R^c))$ such that
\begin{equation*}
u_t \circ \psi_t-u_s \circ \psi_s=\int_s^t {z}'_r \d r=\int_s^t \hat{z}_r\circ\psi_r \d r
\end{equation*}
where $\hat{z}_r\coloneqq{z}'_r \circ\psi^{-1}_r$ and $|\hat{z}|\leq z$ \change{(cf.\cite[Remark 1.1.3 \& Theorem 1.1.2]{AmGiSa08})}.
%Observe the set 
%\begin{equation*}
%B=\{(r,x) \in [0,1]\times\domain:z(r,x)<|\hat{z}(r,x)| \}
%\end{equation*}
%and suppose it has a strictly positive Lebesgue measure. Then it can be approximated with finite unions of disjoint semi-open cuboids. On every such cuboid $[t_1,t_2)\times D$ we have
%\begin{equation*}
%\int_{t_1}^{t_2} \int_D |\hat{z}_t\circ\psi_t|^2 \d x \d t \leq \int_{t_1}^{t_2} \int_D ({z}_t \circ \psi_t)^2 \d x \d t.
%\end{equation*}
%By using the dominated convergence theorem we have the same estimate on the entire set $B$, which is an obvious contradiction.
Thus, we have verified \eqref{eq:first_variational_equality}. This implies
\begin{align*}
\Big|\int_0^\tau \hat{z}_{t+l}\circ\psi_{t+l}-\hat{z}_{s+l} \circ \psi_{s+l}\d l\Big|
\leq\int_0^\tau\int_s^t {w}_{r+l}\circ\psi_{r+l}\d r \d l\change{,}
\end{align*}
\change{for every $s \leq t \in [0,1]$ and $\tau>0$ with $t+\tau \leq 1$,
and the same holds for integration on interval $[-\tau,0]$ for $s-\tau \geq 0$.}
 Taking the limit as $\tau$ tends to zero \change{and using Lebesgue's differentiation theorem \cite[Theorem 3.21]{Fo99}} we conclude that for \change{all} $s \leq t \in [0,1]$ and a.e. $x \in \domain$ we have
\begin{equation*}
\Big|\hat{z}_t\circ \psi_t-\hat{z}_s \circ \psi_s\Big|
\leq\int_s^t {w}_r\circ\psi_r\d r.
\end{equation*}
Thus, we can repeat the above procedure to conclude the existence of $\hat{w} \in L^2((0,1),L^2(\domain,\R^c))$ satisfying
\begin{equation*}
\hat{z}_t\circ\psi_t-\hat{z}_s\circ\psi_s
=\int_s^t \hat{w}_r\circ\psi_r\d r, ~ |\hat{w}|\leq w,
\end{equation*}
from where the claim directly follows.
\end{proof}

In \cite{HeRuWi18} a regularization of the spline path energy by addition of weighted geodesic path energy was \change{necessary} for existence and further analysis of the splines. We follow the analogous approach which seems to be unavoidable also in our model.
\begin{definition}[Regularized spline energy]
Let $\sigma>0$, $m>1+\frac{d}{2}$ be \change{an integer,} and $u \in L^2((0,1),\mathcal{I})$ be an image curve. Then the regularized spline energy is defined by
\begin{align}\label{eq:continuous_spline_energy}
\mathcal{F}^{\sigma}[u]
\coloneqq\inf_{{(v,a,z,w) \in \mathcal{C}[u]}}\int_0 ^1 \int_\domain  {L}[a,a] + \frac{1}{\delta}  w^2 + \sigma \left(L[v,v] +\frac{1}{\delta} z^2 \right) \d x \d t, 
\end{align}
where $\mathcal{C}[u] \subset L^2((0,1), \motionSpace) \times L^2((0,1), \motionSpace) \times L^2((0,1)\times \domain) \times L^2((0,1)\times \domain)$ consists of tuples $(v,a, z, w)$ satisfying
\begin{align}
&v_t(\psi_t(t,x))=\dot{\psi}_t(x), ~ \psi_0(x)=x, \label{eq:first_order_flow}\\
&a_t(\psi_t(x))=\ddot{\psi}_t(x), ~\change{\forall x\in \domain, ~t \in [0,1],} \label{eq:second_order_flow}\\
&\left|\image_t \circ \psi_t-\image_s \circ \psi_s\right| \!\leq\! \int_s^t z_r \circ \psi_r \d r,~\change{\forall s\!\leq\! t \!\in\! [0,1],} \label{eq:first_derivative} \\
& \left| u_{t+\tau} \circ \psi_{t+\tau}\!-\!u_t \circ \psi_t\!-\!u_{s+\tau} \circ \psi_{s+\tau}\!+\! u_s \circ \psi_s \right| 
 \!\leq\!\int_0^\tau \int_s^t  w_{r+l} \circ \psi_{r+l} \d r \d l\change{,} ~ \change{\forall \tau, s\!+\!\tau \!\leq\! t \!+\! \tau \!\in\! [0,1].} \label{eq:second_derivative}
\end{align}
\end{definition}
%The point evaluation of image path $u$ is again possible if the set $\mathcal{C}[\image]$ is non-empty .
\change{Let us observe that similar to \cite[Remark 1]{EfKoPo19a} if $C[\image]$ is non-empty we have an additional regularity of the image curve with $u \in C^1([0,1],\imageSpace)$.}
Motivated by \cite{dB63}, we now define the continuous time spline interpolation for given key frames.
\begin{definition}[Continuous time regularized spline interpolation]\label{def:def_spline}
Let $\{\image_j^I\}_{j=1,\dots,J} \in \imageSpace^J$.
We call a minimizer $u \in L^2((0,1),\mathcal{I})$ of $\mathcal{F}^\sigma$ that satisfies \eqref{eq:constraint}
 a continuous time regularized spline interpolation of $\{\image_j^I\}_{j=1,\dots,J}$ with regularity parameter $\sigma$. 
% Similarly, a minimizer of $\mathcal{F}^\sigma$ is called a regularized spline interpolation of $\{\image_j^I\}_{j=1,\dots,J}$, with regularity parameter $\sigma$.
\end{definition}
If we do not impose any additional constraints, we say the continuous time spline interpolation has natural boundary constraints. Imposing periodic boundary \change{condition is} equivalent to defining the image curve $u_t$ on the sphere $\mathbb{S}^1$ instead of the interval $[0,1]$.
\change{In the case of Hermite boundary conditions we additionally fix the values of velocity terms,  for both the flow and the image intensity, at the initial and the final time point. Thus, we ask for $v_0=v_A, v_1=v_B$ and, in light of Proposition \ref{prop:equivalence_of_approaches} and a differentiation of the left hand side of \eqref{eq:first_variational_equality}, that $\hat{z}_0=\hat z_A$ and $\hat{z}_1=\hat z_B$.}
Note that in the case of the Hermite boundary conditions (also called clamped boundary conditions), we implicitly require $t_1=0$ and $t_{J}=1$, so that $u_0$ and $u_1$ are also prescribed.

%%%%%%%%%%%%%%%%%%%%%%%%%%%%%%%%%%%%%%%%%%%%%%%%%%%%%%%%%%%%%%%%%
%%%%%%%%%%%%%%%%%%%%%%%%%%%%%%%%%%%%%%%%%%%%%%%%%%%%%%%%%%%%%%%%%

\section{Variational Time Discretization}\label{section:time_discrete}
In this section we study the variational time discretization of the time continuous (regularized) spline energy. 
To this end, we pick up the approach of \cite{BeEf14,EfKoPo19a} for the variational time discretization of geodesic energy.
We consider a discrete image curve $\mathbf{u}=(u_0,\dots,u_K)$ 
with $u_k \in \imageSpace$ 
and define a set of admissible deformations 
\begin{equation}\label{eq:admset}
\admset\!\coloneqq\!\{\deformation \in H^{m}(\domain,\domain), ~\det(D\deformation)\!\geq\! \epsilon,~ \deformation\!=\!\Id~ \textrm{on}~ \partial \domain\},
\end{equation}
for a fixed (small) $\epsilon>0$, which consists of  
$C^1(\domain,\domain)$-diffeomorphisms~\cite[Theorem 5.5-2]{Ci88}.
\begin{remark}
The case $\epsilon=0$ is discussed in Remark \ref{remark:epsilon0}.
\end{remark}
\change{Considering} $\imagevec \in \imageSpace^{K+1}$ as time sampling at times $\frac{k}{K}$, $k=0,\dots,K$, of a smooth curve and $\defvec=(\deformation_1,\ldots,\deformation_K)\in\admset^K$ as relative flow ($\deformation_k=\psi_{\frac{k}{K}} \circ \psi^{-1}_{\frac{k-1}{K}}$) and using \change{forward} finite difference approximations we obtain the discrete version of the Eulerian velocity $v_k\coloneqq K(\deformation_k -\Id)$ and $\hat z_k\coloneqq K(\image_{k} \circ \deformation_{k}-\image_{k-1})$ for the discrete material derivative.
Furthermore, by using central finite difference we define the discrete acceleration
\begin{equation}\label{eq:discrete_acceleration_definition}
a_k \coloneqq K^2((\phi_{k+1}-\Id)\circ \phi_k - (\phi_k -\Id))
\end{equation}
and
\begin{align}
\hat w_k\coloneqq& K(\hat z_k \circ \phi_k - \hat z_{k-1})= K^2\left(\image_{k+1}\circ \deformation_{k+1} \circ \deformation_k - 2 \image_k \circ \deformation_k +\image_{k-1} \right) \label{eq:discrete_w_definition}
\end{align}
as the discrete version of the second material derivative.
Following \cite{BeEf14,EfKoPo19a} we consider the discrete path energy
\begin{align*}
\Pathenergy^{K,D}[\imagevec,\defvec]&\!\coloneqq\!
K\sum_{k=1}^{K} \!\int\limits_{\domain}\!\energyDensity_D(D\deformation_k)\!+\! \gamma |D^m \deformation_k|^2  \!+\! \frac{1}{\delta}|u_k\circ\deformation_k \!-\! u_{k-1}|^2 \d x,
\end{align*}
where $\energyDensity_D(B)\!\coloneqq\!\left|B^{sym}-\Id \right|^2$ is a simple elastic energy density.
Then, the discrete counterpart of the spline energy is defined as
\begin{align}
\SplinePathenergy^{K,D}[\imagevec,\defvec]
\!\coloneqq\! \!\frac{1}{K}\!\sum_{k=1}^{K-1}\!\int\limits_{\domain} \energyDensity_A(Da_k) + \gamma |D^m a_k|^2 +\frac{1}{\delta}\left|\hat w_k \right|^2 \d x,\label{eq:DeformationSplinePathenergy}
\end{align}
with the energy density $\energyDensity_A(B)\coloneqq\left|B^{sym}\right|^2$. We note that 
$a_k \in H^m(\domain,\R^d)$ by \cite[\change{Proposition 2.19}]{InKaTo13}.
Finally, the regularized discrete spline energy is given by
\begin{equation*}
\SplinePathenergy^{\sigma,K,D}[\imagevec,\defvec]\coloneqq\SplinePathenergy^{K,D}[\imagevec,\defvec]+\sigma\Pathenergy^{K,D}[\imagevec,\defvec]\,.
\end{equation*}
As in the continuous time model we have interpolation constraints. Let $\change{I^K \coloneqq} (i_1,\dots,i_{J^K})$ be an index tuple \change{with} $2 \leq \change{J^K} \leq K$, \change{$i_j \in \{0,\dots,K\}$ for $j=1,\dots,J^K$}.
We consider a $\change{J^K}$-tuple 
$\imageSpace^K_{f}=({\image}^I_{1},\dots,{\image}^{\change{I}}_{{J^K}})$ and define the set of admissible image vectors
\begin{equation}\label{eq:admissible_images}
\imageSpace^K_{adm}\change{\coloneqq}\{\imagevec \in \imageSpace^{K+1},~ \image_{i_j}={\image}^I_{j},~ j=1,\dots,J^K\}.
\end{equation}
%We further denote $K_j=t_{j+1}-t_j$.
We are now in a position to define discrete splines.
\begin{definition}[\change{Discrete time regularized spline interpolation}]\label{def:discrete_spline}
Let $\imagevec=(u_0, \dots, u_K) \in \imageSpace^{K}_{adm}$. Then we set
\begin{equation}\label{eq:SplinePathenergy}
\SplinePathenergy^{\sigma,K}[\imagevec]:=\inf_{\defvec \in \admset^K}\SplinePathenergy^{\sigma,K,D}[\imagevec,\defvec].
\end{equation}
A \change{discrete time regularized spline interpolation of \\ $\{{\image}^I_{j}\}_{j=1,\dots,J^K}$}  is a discrete $(K+1)$-tuple that minimizes $\SplinePathenergy^{\sigma,K}$ over all discrete paths in $\change{\imageSpace}^K_{adm}$.
\end{definition}

The presented discretization is valid in the case of natural boundary conditions, to which we restrict in \change{further} discussions. We remark that
in the case of periodic boundary conditions we make an identification $K \overset{\scriptscriptstyle\wedge}{=} 0, K+1 \overset{\scriptscriptstyle\wedge}{=} 1$ and  the sum in \eqref{eq:DeformationSplinePathenergy} goes up to $K$.
For the discrete version of Hermite boundary conditions we prescribe $\phi_1=\overline{\phi}_1,  \phi_K=\overline{\phi}_K$, $u_0=\overline{u}_0, u_K=\overline{u}_K$ and $\hat{z}_1=\overline{\hat{z}}_1,\hat{z}_K=\overline{\hat{z}}_K$
for given $\overline{\phi}_1,\overline{\phi}_K\in\deformationSpace$, $\overline{u}_0, \overline{u}_K\in\imageSpace$ and $\overline{\hat{z}}_1,\overline{\hat{z}}_K\in L^2(\domain,\R^c)$.

Next, we follow ideas from \cite{EfKoPo19a} in order to prove the existence of discrete spline interpolations. The following lemma is the analogous result to \cite[Lemma 1]{EfKoPo19a} and we only state it for completeness.
\begin{lemma}\label{lemma:sobolev_upper_bound}
There exists a constant $C$ which only depends on $\domain, m,d, \gamma$ such that
%continuous and motonically increasing function $\theta : \R^+_0 \to \R^+_0$ with $\theta(0)=0$, 
\begin{equation*}
\|\deformation - \Id\|_{H^{m}(\domain)} \leq C \sqrt{C_{\deformation}}
\end{equation*}
for all $\deformation \in \admset$ satisfying
$\int_\domain \energyDensity_D(D\deformation) \!+\! \gamma |D^{m} \deformation|^2 \d x \leq C_{\deformation}$.
%Furthermore $\theta(x) \leq C(x+x^2)^{\frac12}$ for a constant $C>0$.
\end{lemma}
\begin{proof}
An application of the Gagliardo--Nirenberg inequality~\cite{Ni66} yields
\begin{equation}
\Vert\deformation-\Id\Vert_{H^m(\domain)}\leq C(\Vert\deformation-\Id\Vert_{L^2(\domain)}+|\deformation-\Id|_{H^m(\domain)})\,.
\label{eq:mGrowthDisplacement}
\end{equation}
The last term in \eqref{eq:mGrowthDisplacement} is bounded by
\begin{equation}
|\deformation-\Id|_{H^m(\domain)}=|\deformation|_{H^m(\domain)}\leq\sqrt{\tfrac{{C_\deformation}}{\gamma}}\,.
\label{eq:displacementHigherOrderControl}
\end{equation}
%By using the embedding of $H^m(\domain,\domain)$ into $C^{1,\alpha}(\overline\domain,\overline\domain)$ and the uniform boundedness of the minimizing sequence in $L^2(\domain,\domain)$ we get 
%$\Vert\deformation-\Id\Vert_{C^{1,\alpha}(\overline\domain)}
%\leq C+C\sqrt{{C_\deformation}}\,.$ 
To estimate the lower order term on the right hand side we use Korn's and Poincare's inequality and write 
\begin{equation}
\Vert\deformation-\Id\Vert_{L^2(\domain)}\leq C\Vert\varepsilon[\deformation]-\Id\Vert_{L^2(\domain)}\leq C \sqrt{C_\phi}.
\label{eq:KornEstimate}
\end{equation}
Thus, the lemma follows by combining \eqref{eq:mGrowthDisplacement}, \eqref{eq:displacementHigherOrderControl} and \eqref{eq:KornEstimate}.
\end{proof}
\begin{remark}
The analogous result holds for the boundedness of acceleration i.e. \\ $\int_\domain \energyDensity_A(Da)+\gamma |D^{m}a|^2 \d x \leq C_a$ implies $\|a\|_{H^{m}(\domain)}\leq C\sqrt{C_a}$.
\end{remark}
Now, we show the well-posedness of \eqref{eq:SplinePathenergy}.
\begin{proposition}\label{prop:optimal_phi_existence}
For every $K \in \N$ and every image vector $\imagevec=(\image_0,\dots,\image_K) \in \change{\imageSpace^K_{adm}}$ there exists a deformation vector $\defvec=(\deformation_1,\dots,\deformation_K) \in \deformationSpace^K$ such that
\begin{equation*}
\SplinePathenergy^{\sigma,K,D}[\imagevec,\defvec]=\inf_{\tilde{\defvec} \in \deformationSpace^K} \SplinePathenergy^{\sigma,K,D}[\imagevec,\tilde\defvec].
\end{equation*}
\end{proposition}
\begin{proof} 
Let $\{\defvec^j\}_{j \in \N} \in \deformationSpace^K$ be a sequence for which it holds $\lim_{j \to \infty}\!\SplinePathenergy^{\sigma,K,D}[\imagevec,\defvec^j]\!=\! \inf_{\tilde{\defvec} \in \deformationSpace^K} \!\SplinePathenergy^{\sigma,K,D}[\imagevec,\tilde\defvec]$ and $\SplinePathenergy^{\sigma,K,D}[\imagevec,\defvec^j] \leq \overline{\mathbf{F^{\sigma,K}}} \coloneqq \SplinePathenergy^{\sigma,K,D}[\imagevec,\Id^K]$. By Lemma \ref{lemma:sobolev_upper_bound} we have
\begin{equation*}
\|\phi^j_k - \Id\|_{H^m(\domain)} \leq C\sqrt{\change{\dfrac{\overline{\mathbf{F^{\sigma,K}}}}{K}}}, ~ \forall j \in \N,~ k=1,\dots,K.
\end{equation*}
Thus, $\{\defvec^j\}_{j \in \N}$ is uniformly bounded in $H^m(\domain,\domain)^K$. By reflexivity of this space there exists a  subsequence (with the same label) such that $\defvec^j \rightharpoonup \defvec$ in $H^m(\domain,\domain)^K$. By the compact Sobolev embedding, we have $\defvec^j \to \defvec$ in $C^{1,\alpha}(\overline{\domain},\overline{\domain})^K$ \change{for $\alpha \in (0,m-1-\frac{d}{2})$}, which gives us that $\defvec \in \deformationSpace^K$. Analogously, we have boundedness of $\change{\{\accvec^j\}_{j \in \N}}$ in $H^{m}(\domain,\domain)^{K-1}$ and thus a convergent subsequence satisfying  $\accvec^j \rightharpoonup \accvec$ in $H^m(\domain,\domain)^{K-1}$ and $\accvec^j \to \accvec$ in $C^{1,\alpha}(\overline{\domain},\overline{\domain})^{K-1}$. Here, for every $j \in \N$ we have $\accvec^j=(a^j_1,\dots,a^j_{K-1})$ given by \eqref{eq:discrete_acceleration_definition}. From the strong convergence of deformations we have that $a_k=K^2(\phi_{k+1} \circ \phi_k - 2 \phi_k + \Id)$ for all $k=1,\dots,K-1$.
Using weak lower semicontinuity of $H^m$-seminorm and continuity of energy densities we have for all $k$, as $j \to \infty$:
\begin{align}
&\liminf |\deformation^j_k|_{H^{m}} \geq |\deformation_k|_{H^{m}}, ~~ 
\|\energyDensity_D(D\deformation^j_k)\|_{L^1(\domain)} \to \|\energyDensity_D(D\deformation_k)\|_{L^1(\domain)},\nonumber\\
&\liminf |a^j_k|_{H^{m}} \geq |a_k|_{H^{m}}, ~~
\|\energyDensity_A(Da^j_k)\|_{L^1(\domain)} \to \|\energyDensity_A(Da_k)\|_{L^1(\domain)}.\label{eq:R_terms_convergence}
\end{align}
To handle \change{the} rest of the terms we show that for all $u \in \imageSpace$ we have $u \circ \phi^j \to u \circ \phi$ in $\imageSpace$ for $\{\phi^j\}_{j \in \N} \in \deformationSpace$ with $\phi^j \to \phi$ in $C^1(\overline{\domain},\overline{\domain})$ as $j \to \infty$.
To see this 
we approximate $\image$ by smooth functions $\{\tilde\image^i\}_{\change{i \in \N}}\in C^\infty(\overline\domain,\overline\domain)$ with
$\Vert\image-\tilde\image^i\Vert_{\imageSpace}\to 0$.
Then, using the transformation formula we obtain
\begin{align}
&\Vert\image\circ\deformation^j-\image\circ\deformation\Vert_{\imageSpace} \nonumber\\
\leq&
\Vert\image\circ\deformation^j\!-\!\tilde\image^i\circ\deformation^j\Vert_{\imageSpace}\!+\!\Vert\tilde\image^i\circ\deformation\!-\!\image\circ\deformation\Vert_{\imageSpace} 
\!+\!\Vert\tilde\image^i\circ\deformation^j\!-\!\tilde\image^i\circ\deformation\Vert_{\imageSpace} \nonumber\\
\leq&
\Vert\image-\tilde\image^i\Vert_{\imageSpace}\Big(\Vert\det(D\deformation^j)\Vert_{L^\infty(\domain)}^{-\frac12}
\!+\!\Vert\det(D\deformation)\Vert_{L^\infty(\domain)}^{-\frac12}\Big) 
+\Vert D\tilde\image^i\Vert_{L^\infty(\domain)}\Vert\deformation^j-\deformation\Vert_{L^2(\domain)}\,. \label{eq:fixed_u_composition_convergence}
\end{align}
By first choosing~$i$ and then~$j$ we have that this expression converges to $0$. Hence, for every $k=1,\dots,K$ we have 
$\|\image_k \circ \deformation^j_k - \image_{k-1}\|_\imageSpace \to \|\image_k \circ \deformation_k - \image_{k-1}\|_\imageSpace$. Furthermore, via nested application of \eqref{eq:fixed_u_composition_convergence} $\|\image_{k+1} \circ \deformation^j_{k+1} \circ \deformation^j_k - 2 \image_k \circ \deformation^j_k + \image_{k-1}\|_{\imageSpace} \to \|\image_{k+1} \circ \deformation_{k+1} \circ \deformation_k - 2 \image_k \circ \deformation_k + \image_{k-1}\|_{\imageSpace}$ for $k=1,\dots,K-1$, which together with \eqref{eq:R_terms_convergence} finishes the proof.
\end{proof}
In the next step, under suitable conditions, we prove that there exists a minimizing vector in $\imageSpace^K_{adm}$ (see \eqref{eq:admissible_images}) for a fixed deformation vector $\defvec \in \admset^K$.
\begin{proposition}\label{prop:optimal_u_existence}
Let $K \geq 2$, $\imageSpace^K_{f}$ and $\defvec \in \admset^K$ be fixed. Assume that the deformations satisfy, for every $x \in \domain$,
\begin{equation}\label{eq:bound_determinant}
C_{\det}\!\geq\! \det(D\deformation_k(x))\!\geq\! c_{\det}>0, ~k=1,\dots,K.
\end{equation}
Then there exists a vector of images $\imagevec \in \imageSpace^K_{adm}$ such that
\begin{equation*}
\SplinePathenergy^{\sigma,K,D}[\imagevec,\defvec]=\inf_{\change{\mathbf{v} \in \imageSpace^K_{adm}}}\SplinePathenergy^{\sigma,K,D}[\change{\mathbf{v}},\defvec].
\end{equation*}
\end{proposition}
\begin{proof}
Let $\change{\{\imagevec^j\}}_{j \in \N} \in \imageSpace^K_{adm}$ be an approximation sequence such that
$$\lim_{j\to \infty} \SplinePathenergy^{\sigma,K,D}[\imagevec^j,\defvec]=\inf_{\change{\mathbf{v}} \in \imageSpace^K_{adm}}\SplinePathenergy^{\sigma,K,D}[\change{\mathbf{v}},\defvec]\leq\overline{\SplinePathenergy^{\sigma,K,D}}.$$
Here, $\overline{\SplinePathenergy^{\sigma,K,D}}\coloneqq \SplinePathenergy^{\sigma,K,D}[\mathbf{\overline{u}},\defvec]$ represents a finite upper bound for the energy with the vector of images $\mathbf{\overline{u}}$ satisfying $\overline{u}_k=\image^I_1$ for $\change{0 \leq k \leq i_1},$ $\overline{u}_k=\image^I_{j}$ for \change{$i_j \leq k \leq i_{j+1}$ with $1 \leq j \leq J^K-1$ and $\overline{u}_k=u^I_{J^K}$ for $ i_{J^K} \leq k \leq K$}. 
%$\mathbf{\overline{u}}=({u}^I_{1},\dots,{u}^I_{1},{u}^I_{2},\dots,{u}^I_{{J^K-1}},{u}^I_{{J^K}},\dots)$.
 Indeed, we have
\begin{align*}
&\overline{\SplinePathenergy^{\sigma,K,D}}\\
\leq& \sum_{k=1}^{K-1} \frac{1}{K} (\|\energyDensity_A(Da_k)\|_{L^1({\domain})} +\gamma \|a_k\|^2_{H^{m}})
 + \sigma \sum_{k=1}^K K(\|\energyDensity_D(D\deformation_k)\|_{L^1(\domain)}+\gamma \|\deformation_k\|^2_{H^{m}}) +K^3(1+c_{\det}^{-1})^2\sum_{j=1}^{J^K} \|{u}^I_{j}\|_{\imageSpace}^2,
\end{align*}
where we used the transformation formula and \eqref{eq:bound_determinant}. By further use of \eqref{eq:bound_determinant} we have
\begin{align}
&\Vert\image_k^j\Vert_{\imageSpace}\leq\Vert\image_{k+1}^j\circ\deformation_{k+1}-\image_k^j\Vert_{\imageSpace}+\Vert\image_{k+1}^j\circ\deformation_{k+1}\Vert_{\imageSpace}  \leq\sqrt{\tfrac{\delta\overline{\SplinePathenergy^{\sigma,K,D}}}{K}}+c_{\det}^{-\frac{1}{2}}\Vert\image_{k+1}^j\Vert_{\imageSpace}\,,
\nonumber\\
&\Vert\image_{k+1}^j\Vert_{\imageSpace}\!\leq\! C_{\det}^{-\frac12}\Vert\image_{k+1}^j\!\circ\!\deformation_{k+1}\Vert_{\imageSpace} \!\leq\!  C_{\det}^{-\frac12}\left(\!\Vert\image_{k+1}^j\!\circ\!\deformation_{k+1}\!-\!\image_k^j\Vert_{\imageSpace}\!+\!\Vert\image_k^j\Vert_{\imageSpace}\!\right) \!\leq\! C_{\det}^{-\frac12}\!\left(\!\sqrt{\tfrac{\delta\overline{\SplinePathenergy^{\sigma,K,D}}}{K}}\!+\!\Vert\image_{k}^j\Vert_{\imageSpace}\!\right), \label{eq:uniformBoundImagesInduction}
\end{align}
from where we have, by induction, that $\change{\{u^j_k\}_{j \in \N}}$ is uniformly bounded in $\imageSpace$ for every $k=0,\dots,K$. By reflexivity, there exists a subsequence (\change{labeled} in the same way) such that $u^j_k \rightharpoonup u_k$ for some $\imagevec \in \imageSpace^K_{adm}$.
It remains to verify the weak lower semicontinuity of the matching functional, i.e. we have to show that
\begin{align}
\Vert\image_k\circ\deformation_k-\image_{k-1}\Vert_{\imageSpace}^2\leq&\liminf_{j\to\infty}\Vert\image^j_k\circ\deformation_k-\image^j_{k-1}\Vert_\imageSpace^2, \label{eq:lowerSemicontinuityMatching1}\\
\Vert\image_{k+1}\circ\deformation_{k+1}\circ \deformation_k-2\image_k\circ \deformation_k+\image_{k-1}\Vert_{\imageSpace}^2 
\leq&\liminf_{j\to\infty}\Vert\image_{k+1}^j\circ\deformation_{k+1}\circ \deformation_k-2\image_k^j\circ \deformation_k+\image_{k-1}^j\Vert_{\imageSpace}^2, \label{eq:lowerSemicontinuityMatching2}
\end{align}
for every $k=1,\dots,K$ and $k=1,\dots,K-1$, respectively.
To this end, we first show $\image_k^j\circ\deformation_k\rightharpoonup\image_k\circ\deformation_k$ in~$\imageSpace$.
For every $v \in \imageSpace$ the transformation formula yields
\begin{align*}
&\int_{\domain}(\image^j_k\circ\deformation_k-\image_k\circ\deformation_k)\cdot v\d x
=\int_{\domain}(\image^j_k-\image_k)\cdot(v(\det (D\deformation_k))^{-1})\circ\deformation_k^{-1}\d x\,,
\end{align*}
which converges to $0$ since $(v(\det (D\deformation_k))^{-1})\circ\deformation_k^{-1}\in\imageSpace$ due to~\eqref{eq:bound_determinant}.
Hence, $\image_k^j\circ\deformation_k-\image_{k-1}^j\rightharpoonup\image_k\circ\deformation_k-\image_{k-1}$ in~$\imageSpace$,
which readily implies~\eqref{eq:lowerSemicontinuityMatching1} and by applying the same technique in a nested fashion we get \eqref{eq:lowerSemicontinuityMatching2}.
Altogether, we have
\begin{equation*}
\liminf_{j \to \infty} \SplinePathenergy^{\sigma,K,D}[\imagevec^j,\defvec] \geq \SplinePathenergy^{\sigma,K,D}[\imagevec,\defvec],
\end{equation*}
from where the optimality follows.
\end{proof}
We are now in the position to show the existence of discrete \change{time spline interpolations}.
\begin{theorem}[\change{Existence of discrete time spline interpolations}]\label{thm:discrete_spline_existence}
Let $K\geq 2$. Then for every $\imageSpace^K_{f}$
there exists $\imagevec \in \imageSpace^K_{adm}$ such that
\begin{equation*}
\SplinePathenergy^{\sigma,K}[\imagevec]=\inf_{\mathbf{v} \in \imageSpace^K_{adm}} \SplinePathenergy^{\sigma,K}[\mathbf{v}].
\end{equation*}
\end{theorem}
\begin{proof}
Consider a sequence $\{\imagevec^j\}_{j \in \N} \in \imageSpace_{adm}^K $ for which it holds $\lim_{j \to \infty} \SplinePathenergy^{\sigma,K}[\imagevec^j]=\inf_{\mathbf{v} \in \imageSpace^K_{adm}} \SplinePathenergy^{\sigma,K}[\mathbf{v}] \leq\overline{\SplinePathenergy^{\sigma,K}}$, 
%where $\overline{\SplinePathenergy^{\sigma,K}}\coloneqq \SplinePathenergy^{\sigma,K,D}[\overline{\imagevec^I},\Id^K]$ and $\overline{\imagevec^I}$ is a vector of images made by linear interpolation between each two consecutive fixed images (see \eqref{eq:admissible_images}). We have 
%$$\overline{\SplinePathenergy^{\sigma,K}}\leq \left(\frac{\sigma}{\delta} + \frac{2K}{\delta (t^K_{j+1}-t^K_j)^2}\right)\sum_{j=1}^{J^K-1}\|\image^I_{{j+1}}-\image^I_{j}\|^2<\infty,
%$$
%where $t^K_j=\frac{t_j}{K}$.
\change{Here, $\overline{\image}_k=\overline{\image}({\frac{k}{K}},\cdot)$ for $k=0,\dots,K$, where $\overline{\image}$ is a smooth in time curve with $\bar{u}({\frac{i_j}{K}},\cdot)=u^I_j$ for every $j=1,\dots,J^K$.
Then we have
\begin{align}
 \mathbf{F}^{\sigma,K,D}[\overline{\imagevec}^K,\Id^K] 
=&\sigma K\sum_{k=0}^K \int_\domain|\overline{\image}^K_{k+1}-\overline{\image}^K_k|^2 \d x +K^3 \sum_{k=1}^{K-1} \int_\domain |\overline{\image}^K_{k+1}-2\overline{\image}^K_k + \overline{\image}^K_{k-1}|^2 \d x \nonumber\\
\leq & C\left(\int_{\domain} |\overline{\image}|^2_{H^1((0,1))}+|\overline{\image}|^2_{H^2((0,1))} \d x +1\right), \label{eq:finite_upper_bound_discrete}
\end{align}
where the constant is independent of $K$.}
Furthermore, for every $j$ let $\SplinePathenergy^{\sigma,K}[\imagevec^j]=\SplinePathenergy^{\sigma,K,D}[\imagevec^j,\defvec^j]$, \change{by Proposition~\ref{prop:optimal_phi_existence}}.
As in the proof of Proposition~\ref{prop:optimal_phi_existence} we have  
weak convergence of $\change{\{\defvec^j\}_{j \in \N}}$ in $H^{m}(\domain,\domain)^K$ and strong in $C^{1,\alpha}(\overline\domain,\overline\domain)^K$ to some $\defvec \in \deformationSpace^K$. Furthermore, we again have $a^j_k \rightharpoonup a_k$ in $H^{m}(\domain,\domain)$, and strongly in $C^{1,\alpha}(\overline\domain,\overline\domain)$, where $a_k= K^2(\deformation_{k+1} \circ \deformation_k -2\deformation_k + \Id)$ and estimates from \eqref{eq:R_terms_convergence} are \change{satisfied}.
%This already implies the pointwise bound on determinant of derivatives and thus $\defvec \in \admset^K$ and the upper bound needed in \eqref{eq:bound_determinant} is satisfied. 
%This already implies that
%\begin{align}
%&\liminf |\deformation^j_k|_{H^{m}} \geq |\deformation_k|_{H^{m}},~
%\int_{\domain}\energyDensity_D(D\deformation^j_k) \to \int_{\domain}\energyDensity_D(D\deformation_k),\\
%&\liminf |a^j_k|_{H^{m-1}} \geq |a_k|_{H^{m-1}},~
%\int_{\domain}\energyDensity_A(Da^j_k) \to \int_{\domain}\energyDensity_A(Da_k).
%\end{align}
By Proposition~\ref{prop:optimal_u_existence} we may replace $\imagevec^j$ by an energy optimal image vector. Keeping the same label and following the same arguments as in \eqref{eq:uniformBoundImagesInduction} we conclude that $\change{\{\imagevec^j\}_{j \in \N}}$ is uniformly bounded in $\imageSpace$, weakly converging to $\imagevec$.
We are left to verify the estimates
\begin{align*}
\Vert\image_k\circ\deformation_k-\image_{k-1}\Vert_{\imageSpace}^2\leq&\liminf_{j\to\infty}\Vert\image^j_k\circ\deformation^j_k-\image^j_{k-1}\Vert_\imageSpace^2,\\
\Vert\image_{k+1}\circ\deformation_{k+1}\circ \deformation_k-2\image_k\circ \deformation_k+\image_{k-1}\Vert_{\imageSpace}^2 
\leq&\liminf_{j\to\infty}\Vert\image_{k+1}^j\circ\deformation^j_{k+1}\circ \deformation^j_k-2\image_k^j\circ \deformation^j_k+\image_{k-1}^j\Vert_{\imageSpace}^2,
\end{align*}
for every $k=1,\dots,K$ and $k=1,\dots,K-1$, respectively. To that end, it is enough to show $\image^j_k\circ\deformation^j_k \rightharpoonup \image_k\circ\deformation_k$ in $\imageSpace$.
To see this, we first take into account the decomposition
\begin{equation}\label{eq:decomposition_approximation}
\image^j_k\circ\deformation^j_k\!-\!\image_k\circ\deformation_k \!=\! (\image^j_k\circ\deformation^j_k\!-\!\image_k\circ\deformation^j_k)\!+\!(\image_k\circ\deformation^j_k\!-\!\image_k\circ\deformation_k)\,.
\end{equation}
The second term is handled as in \eqref{eq:fixed_u_composition_convergence}.
%To estimate the second term 
%we approximate $\image_k$ by smooth functions $\tilde\image_k^i\in C^\infty(\domain)$ with
%$\Vert\tilde\image-\tilde\image^i\Vert_{\imageSpace}\to 0$.
%Then, using the transformation formula we obtain
%\begin{align*}
%&\Vert\image_k\circ\deformation^j-\image_k\circ\deformation\Vert_{\imageSpace}\\
%\leq&
%\Vert\image_k\circ\deformation^j-\tilde\image_k^i\circ\deformation^j\Vert_{\imageSpace}+\Vert\tilde\image_k^i\circ\deformation^j-\tilde\image_k^i\circ\deformation\Vert_{\imageSpace}
%+\Vert\tilde\image_k^i\circ\deformation-\image_k\circ\deformation\Vert_{\imageSpace}\\
%\leq&
%\Vert\image_k-\tilde\image_k^i\Vert_{\imageSpace}\Big(\Vert\det(D(\deformation^j)^{-1})\Vert_{L^\infty(\domain)}^{\frac12}
%+\Vert\det(D\deformation^{-1})\Vert_{L^\infty(\domain)}^{\frac12}\Big)\\
%&+\Vert D\tilde\image_k^i\Vert_{L^\infty(\domain)}\Vert\deformation^j-\deformation\Vert_{L^2(\domain)}\,.
%\end{align*}
%Finally, by first choosing~$i$ and then~$j$ we have that this expression converges to $0$.
It remains to consider the convergence properties of the first term.
For a test function $v\in\imageSpace$ we obtain using the transformation rule 
\begin{align*}
&\int_\domain(\image^j_k\circ\deformation^j_k-\image_k\circ\deformation^j_k)\cdot v\d x
=\int_\domain(\image^j_k-\image_k)\cdot(v(\det(D\deformation^j_k))^{-1})\circ(\deformation^j_k)^{-1}\d x\,.
\end{align*}
The right hand side converges to~$0$ due to the convergence $(\det(D\deformation^j_k))^{-1}\circ(\deformation^j_k)^{-1}\to \det(D\deformation_k))^{-1}\circ\deformation_k^{-1}$ in $L^\infty(\Omega)$, $v \circ \deformation_k^j \to v \circ \deformation_k$ in $\imageSpace$ and  $\image_k^j \rightharpoonup \image_k$ in $\imageSpace$ for $j\to \infty$. This proves our claim and finally proves the theorem.
\end{proof}
\begin{remark}\label{remark:epsilon0}
The results of this section remain valid for any $\energyDensity_D$ satisfying conditions $(W1)-\change{(W2)}$ from \cite{EfKoPo19a}.
%Furthermore, for large enough $K$ one can show the existence of the discrete spline even if $\epsilon=0$ in the definition of the admissible set. Namely, as in Theorem \ref{thm:discrete_spline_existence} one can construct a finite upper bound $\overline{\mathbf{F}^{\sigma,K}}$ for the spline energy independent of $K$ (see the proof of Theorem \ref{thm:convergence_and_existence}). Then, using Lemma~\ref{lemma:sobolev_upper_bound}, Sobolev embedding theorem and Lipschitz continuity of the determinant we have, for large enough $K$, $\max_{k=1,\dots,K}\|\det(D\phi_k)-1\|_{L^\infty(\domain)}\leq C\sqrt{\frac{\overline{\mathbf{F}^{\sigma,K}}}{K}}<1$, which proves  $\min_{k=1,\dots,K} \det(D\phi_k) \geq c_{\det}>0$ and we can proceed as before.
\change{Furthermore, let us observe the case when $i_j=K \cdot t_j$ and $u_{i_j}=u^I_j$ for $t_j \in [0,1]$ and $u^I_j \in \imageSpace$ for every $j=1,\dots,J$ (cf. \eqref{eq:constraint} and \eqref{eq:admissible_images}). Then for every  large enough $K$ (depending on $u^I_j$ and $t_j$) one can show the existence of the discrete time spline interpolation even if $\epsilon=0$ in the definition of the admissible set \eqref{eq:admset}. Namely, by \eqref{eq:finite_upper_bound_discrete} we have that $\overline{\mathbf{F}^{\sigma,K}}$ is a fixed finite upper bound for the discrete spline energy, independent of $K$. Then, using Lemma~\ref{lemma:sobolev_upper_bound} and Sobolev embedding theorem we have
\begin{align*}
\max_{k=1,\dots,K}\|\phi_k - \Id\|_{C^{1,\alpha}(\overline{\domain})} \leq& C \max_{k=1,\dots,K}\|\phi_k - \Id\|_{H^m({\domain})} 
 \leq C \sqrt{\frac{\overline{\mathbf{F}^{\sigma,K}}}{K}}.
\end{align*}
By Lipschitz continuity of the determinant we have  
\begin{equation*}
\max_{k=1,\dots,K}\|\det(D\phi_k)-1\|_{L^\infty(\domain)}\leq C\sqrt{\frac{\overline{\mathbf{F}^{\sigma,K}}}{K}}<1,
\end{equation*}
for large enough $K$, 
which proves  $\min\limits_{k=1,\dots,K} \det(D\phi_k) \geq c_{\det}>0$ and we can proceed as before.}
\end{remark}
%%%%%%%%%%%%%%%%%%%%%%%%%%%%%%%
\section{Temporal Extension Operators}\label{sec:time_extension}
In this section we define the suitable time extensions of the time discrete quantities from the previous section. This is necessary  to compare discrete and continuous spline functional and to study convergence.

Let $K \geq 2$, $\tau=\frac{1}{K}$, $t^K_k=k\tau$, $t^K_{k\pm \frac12}=(k \pm \frac12)\tau$, $k=0,1,\dots,K$. Furthermore, consider a vector of images $\mathbf{u}^K=(u_0^K,\dots,u^K_K) \in \imageSpace^{K+1}$ and a vector of deformations $\defvec^K=(\phi^K_1,\dots,\phi^K_K) \in \deformationSpace^K $. We first define a \emph{discrete incremental transport path} $y^K$ with $y^K_t=y^K_0(t,\cdot)$ for $t \in [0,t^K_{\frac12}\change{]}$, $y^K_t=y^K_k(t,\cdot)$ for $t \in \change{(}t^K_{k-\frac12},t^K_{k+\frac12}\change{]}$ with $k=1,\dots,K-1$ and $y^K_t=y^K_K(t,\cdot)$ for $t \in \change{(}t^K_{K-\frac12},1]$, where
\begin{align*}
y^K_0(t,\cdot)\change{\coloneqq}&\Id + \tfrac{t}{\tau}(\phi^K_1-\Id)\\
y^K_K(t,\cdot)\change{\coloneqq}&\Id+\tfrac{t-t^K_{K-1}}{\tau}(\phi^K_K-\Id)
\end{align*}
and, for $k=1,\dots,K-1$
\begin{align*}
y^K_k(t,\cdot)\change{\coloneqq}&\tfrac12 (\Id+\phi^K_k)+\tfrac{t-t^K_{k-\frac12}}{\tau}(\phi^K_k-\Id)
+\tfrac{(t-t^K_{k-\frac12})^2}{2\tau^2}\left(\phi^K_{k+1} \circ \phi^K_k -2\phi^K_k +\Id\right).
\end{align*}
%\begin{equation*}
%y^K_k(t,\cdot)\!=\!
%\begin{cases}
%\Id + \tfrac{t}{\tau}(\phi^K_1-\Id), & k=0\\
%\tfrac{\Id+\phi^K_k}{2}+\tfrac{t-t^K_{k-\frac12}}{\tau}(\phi^K_k-\Id)&\\
%+\tfrac{(t-t^K_{k-\frac12})^2}{2\tau^2}\left(\phi^K_{k+1} \circ \phi^K_k -2\phi^K_k +\Id\right),& k=1,\dots,K-1\\
%\Id+\tfrac{t-t^K_{K-1}}{\tau}(\phi^K_K-\Id), & k=K.
%\end{cases}
%\end{equation*}
This can be seen as the cubic Hermite interpolation on intervals $\change{(}t^K_{k-\frac12},t^K_{k+\frac12}]$, and an affine interpolation on $[0,t^K_{\frac12}\change{]}$ and $(t^K_{K-\frac12},1]$. In particular, observe that $y^K_{t^K_{k-\frac12}}=\tfrac{\Id+\deformation^K_k}{2}$ and $y^K_{t^K_{k+\frac12}}=\tfrac{(\Id+\deformation^K_{k+1})\circ \deformation^K_k}{2}$ with the corresponding slopes $\tfrac{\deformation^K_k-\Id}{\tau}$ and $\tfrac{(\deformation^K_{k+1}-\Id)\circ \deformation^K_k}{\tau}$, respectively \change{(cf. Figure \ref{fig:interpol} (left) for a sketch).}

We define the \emph{image extension operator} $\mathcal{U}^K[\mathbf{u}^K,\mathbf{\Phi}^K] \in L^2([0,1],\imageSpace)$ as
$\mathcal{U}^K[\mathbf{u}^K,\mathbf{\Phi}^K](t,x)=\image^K(t,x)$ where
\begin{align*}
u^K_t\circ y^K_t\!\change{\coloneqq}&u^K_0 + \frac{t}{\tau}(u^K_1 \circ \phi^K_1 -u^K_0),~ t \in [0,t^K_{\frac12}\change{]}\\
u^K_t\circ y^K_t\!\change{\coloneqq}&u^K_{K-1} \!+\! \tfrac{t-t^K_{K-1}}{\tau} (u^K_K \circ \phi^K_K \!-\!u^K_{K-1}),~ t \in \change{(}t^K_{K\!-\!\frac12},1]
\end{align*}
and, for $k=1,\dots,K-1$ and $t \in \change{(}t^K_{k-\frac12},t^K_{k+\frac12}\change{]}$
\begin{align}
&u^K_t \circ y^K_t \nonumber\\
\change{\coloneqq}&\tfrac{u^K_{k-1}+u^K_k\circ \phi^K_k}{2} + \tfrac{t-t^K_{k-\frac12}}{\tau} (u^K_k\circ \phi^K_k-u^K_{k-1})
+\tfrac{(t-t^K_{k-\frac12})^2}{2\tau^2} \left(u^K_{k+1}\circ \phi^K_{k+1} \circ \phi^K_k - 2 u^K_k \circ \phi^K_k +u^K_{k-1} \right).\label{eq:image_extension}
\end{align}
%\begin{equation*}
%u^K(t,\cdot)=
%\begin{cases}
%\!&\!u^K_0 + \frac{t}{\tau}(u^K_1 \circ \phi^K_1 -u^K_0), \quad \quad  t \in [0,t^K_{\frac12}]\\
%\!&\!\tfrac{u^K_{k-1}+u^K_k\circ \phi^K_k}{2} + \tfrac{t-t^K_{k-\frac12}}{\tau} (u^K_k\circ \phi^K_k-u^K_{k-1})\\
%\!&\!+\tfrac{(t-t^K_{k-\frac12})^2}{2\tau^2} \left(u^K_{k+1}\circ \phi^K_{k+1} \circ \phi^K_k - 2 u^K_k \circ \phi^K_k +u^K_{k-1} \right), t \in [t^K_{k-\frac12},t^K_{k+\frac12}]\\
%\!&\!u^K_{K-1} + \tfrac{t-t^K_{K-1}}{\tau} (u^K_K \circ \phi^K_K -u^K_{K-1}), \quad \quad t \in [t^K_{K-\frac12},1].
%\end{cases}
%\end{equation*}
This can be seen as blending between the "half-way images" $\tfrac{\image^K_{k-1}+\image^K_k \circ \deformation^K_k}{2}$ and $\tfrac{(\image^K_{k}+\image^K_{k+1} \circ \deformation^K_{k+1})\circ \deformation^K_k}{2}$ along the incremental transport path $y^K$. This is depicted on Figure~\ref{fig:interpol} \change{right}.
\begin{figure*}[htb]
\includegraphics[width=\textwidth]{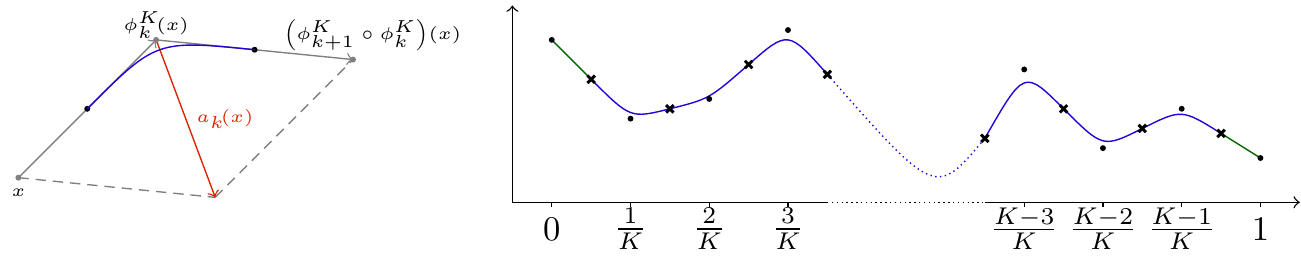}
\caption{\label{fig:interpol} Left: Schematic drawing of the Hermite interpolation 
(blue) on the time interval $\color{blue}[(k-\tfrac12)/K, (k+\tfrac12)/K]$
together with the discrete acceleration $a_k^K(x)$. Right: Image extension $\mathcal{U}^K[\imagevec^K,\mathbf{\Phi}^K](\cdot, x)$ along a path $(\psi_t^K(x))_{t \in [0,1]}$ from the left, 
plotted against time. 
Dots represent the values $u_k^K$, and crosses the "half-way" values 
$\tfrac12 (u_k^K+u_{k-1}^K)$ along the discrete transport path.}
\end{figure*}
The discrete velocity field corresponding to the discrete incremental transport path $y^K_t$ is given by
\begin{equation*}
v^K_0(t,\cdot)\change{\coloneqq}\frac{1}{\tau}(\phi^K_1-\Id), ~~ v^K_K(t,\cdot)\change{\coloneqq}\frac{1}{\tau}(\phi^K_K-\Id),
\end{equation*}
for $t \in [0,t^K_{\frac12}\change{]}$ and $t \in \change{(}t^K_{K-\frac12},1]$, respectively, and 
\begin{align}
%&v^K_0(t,\cdot)=\frac{1}{\tau}(\phi^K_1-\Id), ~~ v^K_K(t,\cdot)=\frac{1}{\tau}(\phi^K_K-\Id)\\
v^K_k(t,\cdot)\change{\coloneqq}\frac{1}{\tau}(\phi^K_k-\Id)\!+\!\tfrac{t-t^K_{k-\frac12}}{\tau^2}\left(\phi^K_{k+1} \circ \phi^K_k -2\phi^K_k +\Id\right), \label{eq:v^K_definition}
\end{align}
for $t \in \change{(}t^K_{k-\frac12},t^K_{k+\frac12}\change{]}$ with $k=1,\dots,K-1$.
%\begin{equation*}
%v^K_k(t,\cdot)=
%\begin{cases}
%\frac{1}{\tau}(\phi^K_1-\Id), & k=0\\
%\frac{1}{\tau}(\phi^K_k-\Id)+\tfrac{t-t^K_{k-\frac12}}{\tau^2}\left(\phi^K_{k+1} \circ \phi^K_k -2\phi^K_k +\Id\right),&k=1,\dots,K-1 \\
%\frac{1}{\tau}(\phi^K_K-\Id), & k=K,
%\end{cases}
%\end{equation*}
The corresponding discrete acceleration is $a_0^K(t,\cdot)=a_K^K(t,\cdot)\change{\coloneqq}0$ for $t \in [0,t^K_{\frac12}\change{]}$ and $t \in \change{(}t^K_{K-\frac12},1]$, respectively, and
%\begin{equation*}
%a^K_k(t,\cdot)\!=\!
%\begin{cases}
%0,\!&\! k\!=\!0,K,\\
%\tfrac{1}{\tau^2}\left(\phi^K_{k+1} \circ \phi^K_k \!-\!2\phi^K_k \!+\!\Id\right),\!&\! k\!=\!1,\dots,K-1.
%\end{cases}
%\end{equation*}
\begin{equation*}
a^K_k(t,\cdot)\change{\coloneqq}
\frac{1}{\tau^2}\left(\phi^K_{k+1} \circ \phi^K_k \!-\!2\phi^K_k \!+\!\Id\right),
%\!&\! k\!=\!1,\dots,K-1.
\end{equation*}
for $t \in \change{(}t^K_{k-\frac12},t^K_{k+\frac12}\change{]}$ with $k=1,\dots,K-1$.
The spatial inverse of the discrete incremental transport path is denoted by $x^K_k(t,\cdot), ~k=0,\dots,K$ which, 
following \cite[Chapter 5]{Ci88} exists if $\max_{k=1,\dots,K} \|D\phi^K_k- \Id \|_{C^0(\domain)} <c$ for a small enough constant $c>0$ (see Section \ref{sec:convergence}\change{, \eqref{eq:deformation_close_to_identity}, \eqref{eq:acceleration_close_to_zero}}). We define the velocity and the acceleration along the incremental transport path by
\begin{align}\label{eq:tilde_quantities}
&\tilde{v}^K_t\change{\coloneqq}v^K_t\circ x^K_t,~ 
\tilde{a}^K_t\change{\coloneqq}a^K_t\circ x^K_t.
\end{align}
Now, the actual discrete flow given as the map $(t,x) \mapsto \psi_t^K(x)$ is defined recursively by
\begin{align}
\psi^K_t\coloneqq& y^K_0(t,\cdot), ~ t \in [0,t^K_{\frac12}\change{]},\nonumber\\
 \psi^K_t\coloneqq&y^K_k(t,\psi^K_{t^K_{k-\frac12}}(\cdot)), ~ t \in \change{(}t^K_{k-\frac12},t^K_{k+\frac12}]\change{,~ k\!=\!1,\!\dots,\!K\!-\!1}\nonumber\\
 \change{\psi^K_t} \change{\coloneqq}& \change{y^K_K(t, \psi^K_{t^K_{K-\frac12}}(\cdot)), ~ t \in (t^K_{K-\frac12},1]}.
 \label{eq:discrete_flow_tilde}
\end{align}
One can directly show that \eqref{eq:discrete_flow_tilde} is well-defined in the sense of equations $\eqref{eq:first_order_flow}- \eqref{eq:second_order_flow}$, i.e.
\begin{align*}
&\dot{\psi}^K_t=\tilde{v}^K_t\circ \psi^K_t, ~ \psi^K_0(x)=x\\
&\ddot{\psi}^K_t=\tilde{a}^K_t\circ\psi^K_t.
\end{align*}
Based on this, the first order scalar weak material derivative of $u^K$ can be defined as the absolute value of the material derivative along the paths $t \mapsto \psi_t^K(x)$ with
%\begin{align*}
%z^K(t,x)=&\tfrac{1}{\tau}|u^K_1 \circ \phi^K_1 -u^K_0|(x^K_0(t,x)), \quad  t \in [0,t^K_{\frac12}]\\
%z^K(t,x)=
%\end{align*}
\begin{align*}
z^K_t\circ y^K_t&\change{\coloneqq}\tfrac{1}{\tau}|u^K_1 \circ \phi^K_1 -u^K_0|,~  t \in [0,t^K_{\frac12}\change{]}\\
z^K_t\circ y^K_t&\change{\coloneqq}\tfrac{1}{\tau} |u^K_K \circ \phi^K_K -u^K_{K-1}|,~ t \in \change{(}t^K_{K-\frac12},1],
\end{align*}
%\begin{equation*}
%z^K_t\circ y^K_t=
%\begin{cases}
%\tfrac{1}{\tau}|u^K_1 \circ \phi^K_1 -u^K_0|,~  t \in [0,t^K_{\frac12})\\
% \tfrac{1}{\tau} |u^K_K \circ \phi^K_K -u^K_{K-1}|,~ t \in [t^K_{K-\frac12},1],
%\end{cases}
%\end{equation*}
and 
\begin{align}
&z^K_t \circ y^K_t 
\change{\coloneqq}\Big|\frac{1}{\tau} (u^K_k\circ \phi^K_k-u^K_{k-1}) +
 \tfrac{t-t^K_{k-\frac12}}{\tau^2} \left(u^K_{k+1}\circ \phi^K_{k+1} \circ \phi^K_k - 2 u^K_k \circ \phi^K_k +u^K_{k-1} \right)\Big|, \label{eq:extension_z_definition}
\end{align}
for $t \in \change{(}t^K_{k-\frac12},t^K_{k+\frac12}\change{]},~k=1,\dots,K-1$.
%\begin{equation*}
%z^K(t,x)=
%\begin{cases}
%\frac{1}{\tau}|u^K_1 \circ \phi^K_1 -u^K_0|(x^K_0(t,x)), \quad  t \in [0,t^K_{\frac12}]\\
%\big|\frac{1}{\tau} (u^K_k\circ \phi^K_k-u^K_{k-1}) +\\
% \tfrac{t-t^K_{k-\frac12}}{\tau^2} \left(u^K_{k+1}\circ \phi^K_{k+1} \circ \phi^K_k - 2 u^K_k \circ \phi^K_k +u^K_{k-1} \right)\big|(x^K_k(t,x)),~ t \in [t^K_{k-\frac12},t^K_{k+\frac12}]\\\frac{1}{\tau} |u^K_K \circ \phi^K_K -u^K_{K-1}|(x^K_K(t,x)), \quad \quad t \in [t^K_{K-\frac12},1],
%\end{cases}
%\end{equation*}
For the second order scalar weak material derivative we have that for $t \in \change{(}t^K_{k-\frac12},t^K_{k+\frac12}\change{]},~k=1,\dots,K-1$
\begin{align*}
w^K_t \circ y^K_t
\change{\coloneqq}\frac{1}{\tau^2} \left|u^K_{k+1}\circ \phi^K_{k+1} \circ \phi^K_k - 2 u^K_k \circ \phi^K_k +u^K_{k-1} \right|,
\end{align*}
and $w^K_t\change{\coloneqq}0$ elsewhere, 
%\begin{equation*}
%w^K(t,x)=
%\begin{cases}
%\frac{1}{\tau^2} \left|u^K_{k+1}\circ \phi^K_{k+1} \circ \phi^K_k - 2 u^K_k \circ \phi^K_k +u^K_{k-1} \right|(x^K_k(t,x)),& t \in [t^K_{k-\frac12},t^K_{k+\frac12})\\
%0,& \textrm{else}.
%\end{cases}
%\end{equation*}
which is the absolute value of the second time derivative of $\mathcal{U}^\change{K}[\imagevec^K,\defvec^K]$ along the path $t \mapsto \psi_t^K(x)$. Indeed, one easily verifies (cf.~\cite[Proposition 9]{EfNe19}) that $z^K_t$ and $w^K_t$ are admissible in the sense of  equations \eqref{eq:first_derivative} and \eqref{eq:second_derivative}, i.e.
\begin{align*}
\left|u^K_t \circ \psi^K_t-u^K_s\circ\psi^K_s\right|
\leq& \int_s^t z^K_r\circ \psi^K_r \d r,\\
\Big|u^K_{t+\tau}\circ \psi^K_{t+\tau}-u^K_t\circ \psi^K_t
-u^K_{s+\tau}\circ \psi^K_{s+\tau}+u^K_s\circ \psi^K_s\Big| 
\leq& \int_0^\tau \int_s^t w^K_{r+l}\circ \psi^K_{r+l} \d l \d r.
\end{align*}
\begin{remark}
For periodic boundary conditions we \change{make} the cubic interpolation definition on $t \in \change{(}t^K_{k-\frac12},t^K_{k+\frac12}\change{]}$ for $k=1,\dots,K$, with the convention $K \overset{\scriptscriptstyle\wedge}{=} 0, K+1 \overset{\scriptscriptstyle\wedge}{=} 1$ and $\change{(}t^K_{K-\frac12},t^K_{K+\frac12}\change{]}= \change{(}t^K_{K-\frac12},1] \cup [0,t^K_{\frac12}\change{]}$.
%This corresponds to observing the time variable on $\mathbb{S}^1$ instead of $[0,1]$.
\end{remark}
Finally, we define the extension of the energy $\SplinePathenergy^{\sigma,K}$ to a functional $\splinepathenergy^{\sigma,K}$ by
\begin{align*}
&\splinepathenergy^{\sigma,K}[\image]
\change{\coloneqq}\change{\inf\limits_{\imagevec^K \in \imageSpace^{K+1}}}\inf\limits_{\defvec^K \in \admset^K} \{\SplinePathenergy^{\sigma,K,D}[\imagevec^K,\defvec^K]: \mathcal{U}^\change{K}[\imagevec^K,\defvec^K]\!=\!\image\}
\end{align*}
if there exists such $\imagevec^K,\defvec^K$ and $+\infty$ else.
The existence of the infimum follows from the continuity of the constraint $\mathcal{U}^\change{K}[\imagevec^K,\defvec^K]=\image$ w.r.t. \change{the weak convergence of $u^K$ and the strong convergence} of $\defvec^K$(cf.~\cite[Lemma 11]{EfNe19}).
%%%%%%%%%%%%%%%%%%%%%%%%%%%%%%%%%%%%%%%%%%%%%%%%%%%%%%%%%%%%%%%

\section{Convergence of Time Discrete Splines}\label{sec:convergence}
In this section we study the convergence of extension of discrete regularized spline energy $\mathcal{F}^{\sigma,K}$ to continuous counterpart $\mathcal{F}^\sigma$
and the convergence of time discrete minimizers to a time continuous minimizer in the sense of Mosco.

We prove the Mosco convergence of energy $\splinepathenergy^{\sigma,K}$ to $\splinepathenergy^\sigma$ in $L^2((0,1)\change{, \imageSpace})$.
\begin{theorem}[Mosco-convergence of the discrete spline energies]\label{thm:mosco_convergence}
Let $\sigma>0$. Then the discrete spline energies $\{\splinepathenergy^{\sigma,K}\}_{K \in \N}$ converge to $\splinepathenergy^{\sigma}$ in the sense of Mosco in the topology $L^2((0,1)\change{, \imageSpace})$ for $K \to \infty$. In explicit
\begin{itemize}
\item[(i)] for every sequence $\{\image^K \in L^2((0,1)\change{, \imageSpace})\}_{K \in \N}$ which converges weakly to $\image \in L^2((0,1)\change{, \imageSpace})$ as $K \to \infty$ it holds $\liminf_{K \to \infty} \splinepathenergy^{\sigma,K}[\image^K] \geq \splinepathenergy^{\sigma}[\image]$,
\item[(ii)] for every $\image \in L^2((0,1)\change{, \imageSpace})$ there exists a sequence $\{\image^K\}_{K \in \N}$ such that $u^K \to u$ in $L^2((0,1)\change{, \imageSpace})$ as $K\!\to \!\infty$
and $\limsup_{K \to \infty} \splinepathenergy^{\sigma,K}[\image^K] \leq \splinepathenergy^{\sigma}[\image]$.
\end{itemize}
\end{theorem}
\begin{remark}
The above result holds for any choice of $\energyDensity_D$ satisfying assumption $(W1)\!-\!(W3)$ from \cite{EfKoPo19a}, though we restrict ourselves to the special case $\energyDensity_D\!=\!|A^{sym}-\Id|^2$.
\end{remark}
\begin{remark}
In the case of periodic boundary conditions, the same applies in the topology $L^2(\mathbb{S}^1\change{, \imageSpace})$. The proof requires minor alterations implied by already indicated changes in energy and interpolation.
\end{remark}
\subsubsection*{Proof of the $\liminf$ estimate.}
Suppose we have a sequence $\{\image^K \in L^2((0,1),\imageSpace)\}_{K \in \N}$ such that $\image^K \rightharpoonup \image$ in that space as $K \to \infty$, and we suppose that $\splinepathenergy^{\sigma,K}[\image^K] < \overline{\splinepathenergy} < \infty$.
By definition for every $K$ large enough an image vector $\imagevec^K \in \imageSpace^{K+1}$ and a corresponding optimal vector of deformations $\defvec^K \in \admset^K$ exist, such that
\begin{equation*}
\image^K=\mathcal{U}^\change{K}[\imagevec^K,\defvec^K], ~ \splinepathenergy^{\sigma,K}[\image^K]=\SplinePathenergy^{\sigma,K,D}[\imagevec^K,\defvec^K].
\end{equation*}
For the image vectors $\imagevec^K$ and the deformation vectors $\defvec^K$ we define the discrete velocity and acceleration quantities as in Section \ref{section:time_discrete} and their extensions as in 
Section \ref{sec:time_extension}.
Using Lemma~\ref{lemma:sobolev_upper_bound} we have
\begin{align}\label{eq:deformation_close_to_identity}
\!&\!\max_{k=1,\dots,K}\!\|\deformation^K_k - \Id\|_{C^{1,\alpha}(\overline{\domain})} \!\leq\! CK^{-\frac12},\\
&\max_{k=1,\dots,K-1}\!\|a^K_{k}\|_{C^{1,\alpha}(\overline{\domain})} \leq CK^{\frac12},\label{eq:acceleration_close_to_zero}
\end{align}
which implies that $y^K_k$ is invertible for any $k=0,\dots,K$ and any $K$ large enough. Thus, we are able to define all the temporal extended quantities introduced in the previous section. 
Furthermore, \change{for every $t \in [0,1]$ we have that}  $y^K_k(t,\cdot)$ is in $C^{1,\alpha}(\overline{\domain},\overline{\domain})$ \change{and it is bounded uniformly in $t$,}
and by \cite[Theorem 2.1]{BoHaSt05} the same holds for $x^K_k(t,\cdot)$ with an estimate
\begin{equation}
\|x^K_k(t,\cdot)\|_{C^{1,\alpha}(\overline{\domain})} \!\leq\! C(1\!+\!\max_{k=1,\dots,K}\! \|\deformation^K_k - \Id\|_{C^{1,\alpha}(\overline{\domain})}). \label{eq:bound_on_x}
\end{equation}
In what follows, the quantities $\hat z^K_k$ and $\hat w^K_k$ are time discrete quantities from Section \ref{section:time_discrete}, while $z^K_t$ and $w^K_t$ are their time extensions from Section \ref{sec:time_extension}.
The estimate \eqref{eq:bound_on_x} together with locally Lipschitz property of the determinant function gives
\begin{align*}
\lim_{K \to \infty} \int_0^1 \int_{\domain} (w^K_t)^2 \d x \d t
=&\lim_{K \to \infty}\sum_{k=1}^{K-1}\int_\domain\int_{t^K_{k-\frac12}}^{t^K_{k+\frac12}} \change{|}\hat w^K_k(x^K_k(t,x)\change{|}^2 \d t\d x\\
=&\lim_{K \to \infty}\sum_{k=1}^{K-1}\int_\domain\int_{t^K_{k-\frac12}}^{t^K_{k+\frac12}} \change{|}\hat w^K_k(x)\change{|}^2 \det(Dy^K_k(t,x)) \d t\d x\\
=&\lim_{K \to \infty} \frac{1}{K} \sum_{k=1}^{K-1} \int_{\domain} \change{|}\hat w^K_k\change{|}^2 \d x.
\end{align*}
% Here we used the Riemannian approximation of the integral.
\change{Using the same ideas, together with $|z^K_t - |\hat{z}^K_k \circ x^K_{k,t}||\leq \frac{1}{K} |\hat{w}^K_k \circ x^K_{k,t}|$ (cf. \eqref{eq:extension_z_definition}) on every subinterval we have}
% with approximation of the time integrals first in the initial and then in the final points of intervals $[t^K_{k-\frac12},t^K_{k+\frac12}]$ we have
\begin{equation*}
\lim_{K \to \infty} \!\int_0^1\! \!\int_{\domain}\! (z^K_t)^2 \d x \d t\!=\!\lim_{K \to \infty} \frac{1}{K} \sum_{k=1}^K \!\int_\domain\! \change{|}\hat z^K_k\change{|}^2 \d x.
\end{equation*}
This implies the uniform boundedness of $\{z^K\}_{K \in \N}$ and $\{w^K\}_{K \in \N}$ in $L^2((0,1) \times \domain)$ and by reflexivity, the existence of weakly convergent subsequences (with the same labeling) to $z$ and $w$, respectively. Then by weak lower semicontinuity we have
\begin{align}
&\|z\|^2_{L^2((0,1)\times \domain)} \leq \liminf_{K \to \infty} \|z^K\|^2_{L^2((0,1)\times \domain)},~~~ 
\|w\|^2_{L^2((0,1)\times \domain)} \leq \liminf_{K \to \infty} \|w^K\|^2_{L^2((0,1)\times \domain)}. \label{eq:wlsc_z_w}
\end{align}
Using Korn's and Poincare's inequality we get
\begin{align*}
&\int_0^1 \int_\domain \change{|}a^K_\change{t}\change{|}^2 \d x \d t = \sum_{k=1}^{K-1} \int_{t^K_{k-\frac12}}^{t^K_{k+\frac12}}\int_\domain \change{|}a^K_k(t, x)\change{|}^2 \d x \d t  
\leq \frac{C}{K} \sum_{k=1}^{K-1} \int_\domain {W_A(Da^K_k)} \d x \leq C \overline{\splinepathenergy},\\
&\int_0^1 \int_\domain |D^{m} a^K_\change{t}|^2 \d x \d t
\!=\!\sum_{k=1}^{K-1}\! \int_{t^K_{k-\frac12}}^{t^K_{k+\frac12}}\!\int_\domain\! |D^{m}a^K_k |^2 \d x \d t
\leq \sum_{k=1}^{K-1} \frac{1}{K} \int_\domain |D^{m} a^K_k|^2 \d x \leq C \overline{\splinepathenergy}.
\end{align*}
The analogous estimates hold for $v^K$ and $D^m v^K$, \change{where we additionally use $\left|v^K_t-|v^K_k|\right| \leq \frac{1}{K} |{a}^K_k|$ on every subinterval (cf. \eqref{eq:v^K_definition}).}
%by using approximation of Riemannian integral at the initial and final points of intervals of size $K^{-1}$.
Hence, we have that $\{v^K\}_{K \in \N}$ and $\{a^K\}_{K \in \N}$ are uniformly bounded in $L^2((0,1),\motionSpace)$ and they have corresponding weak limits $v$ and $a$ in that space.

Furthermore, we compute the Taylor expansion of
$\energyDensity_A((t-t^K_{k-\frac12})^2Da^K_k)$ at $t^K_{k-\frac12}$, evaluated at $t=t^K_{k+\frac12}$ to get
\begin{align}
\frac{1}{K^4}\energyDensity_A(Da^K_k)&=\frac{1}{2K^4} D^2\energyDensity_A(\mathbf{0})(Da^K_k,Da^K_k) +r_{a,k}^K 
=\frac{1}{K^4}\tr(\varepsilon[a^K_k]^2)+r_{a,k}^K. \label{eq:taylor_exp1}
\end{align}
For the remainder term we have $r_{a,k}^K=\bigO(K^{-6}|Da^K_k|^3)$ and
by using Lemma \ref{lemma:sobolev_upper_bound} and \eqref{eq:acceleration_close_to_zero}
\begin{align}
&\sum_{k=1}^{K-1} K^3 \int_{\domain} r^K_{a,k} \d x 
\leq \frac{C}{K^{2}} \!\max_{k=1,\dots,K\!-\!1}\!\|a^K_k\|_{C^1(\overline{\domain})}\!\frac{1}{K}\!\sum_{k=1}^{K\!-\!1} \|a^K_k\|^2_{H^{m}(\domain)}
\!\leq\! CK^{-\frac32}\overline{\splinepathenergy}.\label{eq:taylor_exp_r}
\end{align}
Then we use weak lower semicontinuity of the energy  to write
\begin{align}
\liminf_{K \to \infty}\frac{1}{K} \sum_{k=1}^{K-1} \int_\domain\energyDensity_A(Da^K_k) + \gamma |D^{m} a^K_k|^2 \d x 
=&\liminf_{K \to \infty} \frac{1}{K} \sum_{k=1}^{K-1} \int_\domain \tr(\varepsilon[a^K_k]^2) + \gamma |D^{m} a^K_k|^2 \d x \nonumber\\
=&\liminf_{K \to \infty} \int_0^1\int_\domain \tr(\varepsilon[a^K_\change{t}]^2) + \gamma |D^{m} a^K_\change{t}|^2 \d x \d t \nonumber\\
\geq & \int_0^1\int_\domain  \tr(\varepsilon[a]^2) + \gamma |D^{m} a|^2 \d x \d t. \label{eq:wlsc_a}
\end{align}
Analogous Taylor expansion arguments give
\begin{align}
\liminf_{K \to \infty}K \sum_{k=1}^{K} \int_\domain\energyDensity_D(D\deformation^K_k) + \gamma |D^m \deformation^K_k|^2 \d x 
=&\liminf_{K \to \infty} \int_0^1\int_\domain \tr(\varepsilon[v^K_\change{t}]^2) + \gamma |D^m v^K_\change{t}|^2 \d x \d t \nonumber\\
\geq & \int_0^1\int_\domain \tr(\varepsilon[v]^2) + \gamma |D^m v|^2 \d x \d t. \label{eq:wlsc_v}
\end{align}
\change{Altogether, \eqref{eq:wlsc_z_w}, \eqref{eq:wlsc_a} and \eqref{eq:wlsc_v} give 
\begin{equation*}
\liminf_{K \to \infty}\splinepathenergy^{\sigma,K}[\image^K]=\liminf_{K \to \infty}\SplinePathenergy^{\sigma,K,D}[\imagevec^K, \defvec^K] \geq \splinepathenergy^{\sigma}[\image].
\end{equation*}}

\change{We are left to show that the limit objects $v,a,z,w$ are indeed corresponding quantities for the image curve $\image$.}
\change{First,} let us observe that by using \eqref{eq:bound_on_x} (cf. also \cite[Proposition 1.2.4,1.2.7]{Fi17}) we have
\begin{align}
&\|\tilde{v}^K_t\|_{C^{1,\alpha}(\overline{\domain})} 
 \leq C \|v^K_t\|_{C^{1,\alpha}(\overline{\domain})}(1+\max_{k=1,\dots,K} \|\deformation^K_k - \Id\|_{C^{1,\alpha}(\overline{\domain})})\change{,} \label{eq:estimate_on_composition_vel}\\
&\|\tilde{a}^K_t\|_{C^{1,\alpha}(\overline{\domain})}
\leq C \|a^K_t\|_{C^{1,\alpha}(\overline{\domain})}(1+\max_{k=1,\dots,K} \|\deformation^K_k - \Id\|_{C^{1,\alpha}(\overline{\domain})}), \label{eq:estimate_on_composition_acc}
\end{align}
where $\tilde{v}^K_t$ and $\tilde{a}^K_t$ are introduced in \eqref{eq:tilde_quantities}.
\change{This implies that} we have the uniform boundedness of $\{\tilde{v}^K\}_\change{{K \in \N}}$ in $L^2((0,1),\motionSpace)$ and in $L^2((0,1),C^{1,\alpha}(\overline{\domain},\R^d))$ by \eqref{eq:estimate_on_composition_vel}.
Then by \cite[Theorem~6]{EfNe19} and H\"older's inequality we can infer that $\{\psi^K\}_{K \in \N}$ is uniformly bounded in \\ $C^{0,\frac12}([0,1],C^{1,\alpha}(\overline{\domain},\overline{\domain}))$.
Furthermore, by compact embedding of H\"older spaces, we have for some $\min\{\frac12,\alpha\}>\beta >0$ that $\psi^K \to \psi$ in $C^{0,\beta}([0,1],C^{1,\beta}(\overline{\domain},\overline{\domain}))$.
To show that $\psi$ is indeed the solution corresponding to $v$, we consider \change{${\psi}^{v^K}$, the solution} corresponding to $v^K$.
By weak continuity of the solution operator mapping velocities to the flows~\cite[Theorem~6]{EfNe19} we have $\change{{\psi}^{v^K}} \to \tilde{\psi}$ in $C^0([0,1]\times \overline{\domain})$.
Furthermore\change{, by}~\cite[Remark~7]{EfNe19} we have the Lipschitz continuity of the solution operator and using the spatial Lipschitz property of $v^K_k$ together with \eqref{eq:deformation_close_to_identity}
we have $\psi^K-\change{{\psi}^{v^K}} \to 0$ in $C^0([0,1] \times \overline\domain)$ which finally confirms $\psi=\tilde{\psi}$.

To show that the equation 
$\ddot{\psi}_t=a_t \circ \psi_t$ is satisfied 
first observe that the equation $\ddot{\psi}^K_t=\tilde{a}^K_t\circ \psi^K_t$ ensures the uniform boundedness of $\{\ddot{\psi}^K\}_\change{{K\in \N}}$ in the space $L^2((0,1),C^{1,\alpha}(\overline{\domain},\overline{\domain}))$.
To this end we used the uniform boundedness of $\{\tilde{a}^K\}_\change{{K\in \N}}$ in \change{the space} \\ $L^2((0,1),C^{1,\alpha}(\overline{\domain},\overline{\domain}))$(cf.~\eqref{eq:estimate_on_composition_acc}) 
and $\{\psi^K\}_\change{{K\in \N}}$ in \\ $C^0([0,1],C^{1,\alpha}(\overline{\domain},\overline{\domain}))$.
This \change{together with the previous paragraph} implies that \change{$\{\psi^K\}_\change{{K\in \N}}$ is uniformly bounded in $H^2((0,1),C^{1,\alpha}(\overline{\domain},\overline{\domain}))$ 
% and converges weakly to $\bar{\psi} \!\in\! H^1((0,1),C^{1,\beta}(\overline{\domain},\overline{\domain}))$ 
% and strongly in $C^{0,\beta}([0,1],C^{1,\beta}(\overline{\domain},\overline{\domain}))$ for $\min\{\frac12,\alpha\}>\beta>0$. 
% Thus, $\bar{\psi}=\dot{\psi}$ since we already established that $\psi^K \to \psi$ in $C^{0,\beta}([0,1],C^{1,\beta}(\overline\domain,\overline{\domain}))$.
and converges weakly to $\psi$ in $H^2((0,1),C^{1,\beta}(\overline{\domain},\overline{\domain}))$ for $0<\beta<\frac12 \min\{\frac12,\alpha\}$.}
\change{Altogether}, we have that $\psi \in H^2((0,1),C^{1,\beta}(\overline{\domain},\overline{\domain}))$ and
%and $\ddot{\psi}=a^*_t\circ\psi_t$ where $a^*_t=\ddot{\psi}_t\circ \psi^{-1}_t$. 
\begin{equation*}
\Big\vert \int_\domain \int_0^1 \left(\tilde{a}^K_t \circ \psi^K_t(x) - \ddot{\psi}_t(x) \right) \theta_t(x) \d t \d x \Big\vert \to 0,
\end{equation*}
for all $ \theta \in C^{\infty}_c((0,1) \times \domain)$. Hence, it sufficies to verify $\tilde{a}^K \circ \psi^K \rightharpoonup a \circ \psi$ in $L^2((0,1)\times \domain)$. 
Since we already have $a^K \circ \psi^K \rightharpoonup a \circ \psi$ in $L^2((0,1)\times \domain)$, we conclude the proof by checking that $\tilde{a}^K \circ \psi^K - a^K \circ \psi^K \to 0$ in $L^2((0,1),C^0(\overline{\domain},\overline{\domain}))$.
Indeed
\begin{align*}
\|\tilde{a}^K \circ \psi^K - a^K \circ \psi^K\|^2_{L^2((0,1),C^0(\overline{\domain})}
\leq&C\int_0^1 \|\tilde{a}^K_t\circ \psi^K_t - a^K_t\circ\psi^K_t\|^2_{C^0(\overline{\domain})} \d t \\
 \leq& C \sum_{k=1}^K \int_{t^K_{k-1}}^{t^K_k} \|\tilde{a}^K_t\circ\psi^K_t - a^K_t\circ\psi^K_t \|^2_{C^0(\overline{\domain})} \d t \\
\leq &  C \sum_{k=1}^K \int_{t^K_{k-1}}^{t^K_k} \|a^K_k(t,\cdot)\|^2_{C^1(\overline{\domain})} \|y^K_k (t,\cdot) - \Id\|^2_{C^0(\overline{\domain})}\d t \\
\leq & C\|a^K\|^2_{L^2((0,1),C^1(\overline\domain))} \max_{k=1,\dots,K}  \|\deformation^K_k - \Id\|^2_{C^0(\overline{\domain})}\\
\leq & CK^{-1} \|a^K\|^2_{L^2((0,1),C^1(\overline\domain))},
\end{align*}
where we used the Lipschitz property of $a^K_t$, the transformation formula and finally \eqref{eq:deformation_close_to_identity}.

In order to show that $z$ and $w$ are indeed scalar weak material derivatives of $u$, first observe that from the weak convergence $u^K \rightharpoonup u$ and the strong convergence $\psi^K \to \psi$ 
by similar approximation arguments as for \eqref{eq:decomposition_approximation}  we obtain $u^K \circ \psi^K \rightharpoonup u \circ \psi$ in $L^2((0,1) \times \domain)$ and analogously $z^K \circ \psi^K \rightharpoonup z \circ \psi$ 
and $w^K \circ \psi^K \rightharpoonup w \circ \psi$.
Next, note that for $s,t \in [0,1]$ we obtain
\begin{align*}
\int_{\domain} \change{|}\image^K_t\circ \psi^K_t-\image^K_s\circ \psi^K_s\change{|}^2 \d x 
\leq |t-s| \left|\int_s^t \int_{\domain} (z^K_r\circ \psi^K_r)^2 \d x \d r \right| 
\leq C |t-s| \|z^K\|^2_{L^2((0,1) \times \domain)},
\end{align*}
where we used H\"older's inequality, the transformation formula and uniform boundedness of $(\psi^K)^{-1}$. Thus, $\{u^K \circ \psi^K\}_{K \in \N}$ is a subset of 
\begin{equation*}
A_{\frac12,L}:=\{u \in L^2((0,1),\imageSpace): \|u_t-u_s\|_{L^2(\domain)} \leq L|t-s|^{\frac12}\},
\end{equation*} for some finite $L>0$. Then, by the weak closedness of $A_{\frac12,L}$ we obtain $\image \circ \psi  \in A_{\frac12,L}$.
%Assume there exist $s <t \in [0,1]$ such that the set
%\begin{equation*}
%B:=\{ x \in \domain:~\vert \image_t\circ\psi_t(x)- \image_s\circ\psi_s(x) \vert > \int_s^t z_r\circ\psi_r(x) \d r \}
%\end{equation*}
%has positive Lebesgue measure.
\change{Then for every $\tilde{\domain} \subset \domain$ the} functional $\image \mapsto \int_\change{\tilde{\domain}} |\image_t(x)-\image_s(x)| \d x$ is continuous because the point evaluation in time is continuous 
and \change{since it is} convex on $ A_{\frac12,L}$ \change{this} implies weak lower semicontinuity
and we obtain
\begin{align*}
\int_\change{\tilde{\domain}} |\image_t\circ \psi_t-\image_s\circ\psi_s| \d x 
\leq& \liminf_{K \to\infty} \int_\change{\tilde{\domain}} |\image^K_t\circ \psi^K_t-\image^K_s\circ \psi^K_s| \d x \\
\leq& \liminf_{K \to \infty} \int_\change{\tilde{\domain}} \int_s^t z^K_r\circ\psi^K_r \d r \d x\\ 
=& \int_\change{\tilde{\domain}} \int_s^t z_r\circ \psi_r \d r \d x.
\end{align*}
%where in the last equality we used weak convergence and test function $\Id_{B}$. This yields and obvious contradiction proving that $z$ is a scalar weak material derivative for $\image$.
\change{Since this holds for every $\tilde{\domain} \subset \domain$ one obtains that $z$ is the first order scalar material derivative for $\image$.}
We prove that $w$ is the second scalar weak material derivative for $\image$ in an analogous way.
This finally finishes the proof of the $\liminf$-inequality.
\vspace*{0.3cm}

As a preparatory step for the proof of the $\limsup$-estimate we state a corollary of the preceding proof. 
\begin{proposition}\label{prop:optimal_tuple}
For $u \in L^2([0,1],\imageSpace)$ with $\splinepathenergy^\sigma[u] < \infty$ there exists an optimal tuple $(v,a,z,w) \in \mathcal{C}[u]$ such that 
\begin{equation*}
\splinepathenergy^\sigma[u]=\int_0 ^1 {L}[a,a] + \frac{1}{\delta} w^2 + \sigma \left(L[v,v] +\frac{1}{\delta} z^2 \right) \d x \d t.
\end{equation*}
\end{proposition}
\begin{proof}
The functional $\splinepathenergy^\sigma$ is coercive by Korn's inequality and Gagliardo-Nirenberg interpolation estimate and it is clearly weak lower semicontinuous. Since $\mathcal{C}[\image]$ is a subset of a reflexive Banach space, then we just have to show weak closedness of the set.
This is verified as above.
\end{proof}

\subsubsection*{Proof of the $\limsup$ estimate and the construction of a recovery sequence.}
Consider an image curve $u \in L^2([0,1],\imageSpace)$ with finite regularized spline energy. Then, the previous proposition guarantees the existence of  an associated optimal velocity field, an acceleration field and the first and second order weak material derivatives, denoted by $(v,a,z,w)\in\mathcal{C}[u]$, respectively, i.e. 
\begin{align*}
\mathcal{F}^{\sigma}[u]\!=\!
\!\int_0^1\! \!\int_\domain\! {L}[a,a] \!+\! \frac{1}{\delta} w^2 + \sigma \!\left(\!L[v,v] \!+\!\frac{1}{\delta} z^2 \!\right)\! \d x \d t.
\end{align*}
We define 
\begin{equation}\label{eq:repre_limsup}
\deformation^K_k\coloneqq \psi_{t^K_{k-1},t^K_k}=\psi_{t^K_k}\circ\psi^{-1}_{t^K_{k-1}}, ~ k=1,\dots,K,
\end{equation}
where $\psi_t$ is the flow associated with velocity $v$ and $\psi_0=\Id$.
We have 
\begin{align}
\max_{k=1,\dots,K} \|\deformation_{k}^K - \Id\|_{C^1(\overline{\Omega})} 
\leq& \sup_{|s-t|\leq \frac{1}{K}} \|\psi_{s,t}-\Id\|_{C^1(\overline{\Omega})} \label{eq:lim_sup_phi_close_to_id}\\
\leq&\sup_{\vert t-s\vert\leq K^{-1}} C\left\vert\int_s^t\Vert v_r \circ \psi_r \Vert_{H^m(\Omega)}\d r\right\vert \nonumber\\
\leq&\sup_{\vert t-s\vert\leq K^{-1}} C\left\vert\int_s^t\Vert v_r \Vert_{H^m(\Omega)}\d r\right\vert \nonumber\\
\leq& CK^{-\frac{1}{2}}\!\sup_{\vert t-s\vert\leq K^{-1}}\! \left\vert\int_s^t\Vert v_r\Vert_{H^m(\Omega)}^2\d r\right\vert^\frac{1}{2}\!\leq\! CK^{-\frac12}\sqrt{\mathcal{F}^{\sigma}[u]},\nonumber
\end{align}
by Lemma \ref{lemma:sobolev_upper_bound} and Cauchy's inequality. For the second inequality we used \cite[Lemma 3.5]{BrVi17} which states that
\begin{equation}\label{eq:H^m_composition_estimate}
\|v_r \circ \psi_r\|_{H^m(\domain)} \leq C \|v_r\|_{H^m(\domain)}.
\end{equation}
Thereby, for $K$ large enough we have $\defvec \in \deformationSpace^K$ and we are in a position to define 
\begin{equation*}
u^K\change{\coloneqq}\mathcal{U}^\change{K}[\imagevec^K,\defvec^K],~ \imagevec^K(\cdot)\change{\coloneqq}(u(t^K_0,\cdot),\dots,u(t^K_K,\cdot)),
\end{equation*}
where the point time evaluation is possible since $u \in C^\change{{1}}([0,1],\imageSpace)$ (cf. \cite[Remark 1]{EfKoPo19a}).

In what follows we present more detailed arguments for the second order terms, while the arguments for the first order terms follow as in \cite[Theorem 14]{EfNe19}.

First, we are able to \change{relate} the discrete and the continuous second order material derivative by
\begin{align*}
&\int_\domain |\image^K_{k+1} \circ \deformation^K_{k+1} \circ \deformation^K_k - 2 \image^K_k \circ \deformation^K_k + \image^K_{k-1}|^2 \d x \\
=&\int_\domain |u_{t^K_{k+1}}\circ\psi_{t^K_{k-1},t^K_{k+1}} - 2u_{t^K_k}\circ \psi_{t^K_{k-1},t^K_{k}} + u_{t^K_{k-1}}|^2 \d x\\
=&\int_\domain |u_{t^K_{k+1}}\!\circ\! \psi_{t^K_{k+1}} \!-\! 2u_{t^K_k}\!\circ\!\psi_{t^K_k} \!+\! u_{t^K_{k-1}}\!\circ\! \psi_{t^K_{k-1}}|^2
\cdot \det (D\psi_{t^K_{k-1}}) \d x \\
\leq& \int_\domain \left(\int_{t^K_{k-1}}^{t^K_k} \int_0^{\frac{1}{K}} w_{r+s} \circ \psi_{r+s} \d r \d s \right)^2 \det (D\psi_{t^K_{k-1}}) \d x\\
\leq&\frac{1}{K^2} \int_0^{\frac{1}{K}} \int_{t^K_{k-1}}^{t^K_k} \int_\domain w_{r+s}^2  \det (D\psi_{r+s,t^K_{k-1}}) \d x \d s \d r \\
\leq&\frac{1}{K^2} (1+CK^{-\frac{1}{2}})  \int_0^{\frac{1}{K}} \int_{t^K_{k-1}}^{t^K_k} \int_\domain w_{r+s}^2 \d x \d s \d r.
\end{align*}
Here, we first used \eqref{eq:repre_limsup} and the transformation formula for the first and the second equality, respectively, then the definition of the second order material derivative \eqref{eq:second_central_variational_inequality} for the first inequality, and the Cauchy-Schwarz inequality and \eqref{eq:lim_sup_phi_close_to_id} in the last two estimates. 
%The first inequality follows from the defining variational inequality of the second order material derivative.
Summing the above expressions over $k=1,\ldots, K-1$, we obtain
\begin{align}
K^3 \sum_{k=1}^{K-1} \int_\domain |\image^K_{k+1} \circ \deformation^K_{k+1} \circ \deformation^K_k - 2 \image^K_k \circ \deformation^K_k + \image^K_{k-1}|^2 \d x 
\leq&K(1+CK^{-\frac{1}{2}}) \int_0^{\frac{1}{K}}\int_0^{1-\frac{1}{K}} \int_\domain w^2_{r+s} \d x \d s \d r \nonumber\\
\leq& K(1+CK^{-\frac{1}{2}}) \int_0^{\frac{1}{K}} \int_0^{1} \int_\domain w^2_t \d x \d t \d r \nonumber\\
\leq & (1+CK^{-\frac{1}{2}}) \int_0^{1} \int_\domain w^2_t \d x \d t. \label{eq:lim_sup_w_bound}
\end{align}
Next, we express $a^K_k$ in terms of $a$:
\begin{align*}
a_k^K &= K^2(\deformation^K_{k+1} \circ \deformation^K_k - 2 \deformation^K_k + \Id)\\
&=K^2(\psi_{t^K_{k-1},t^K_{k+1}} - 2\psi_{t^K_{k-1},t^K_{k}}+\psi_{t^K_{k-1},t^K_{k-1}})\\
&=K^2 \left(\int_{t^K_k}^{t^K_{k+1}} \dot{\psi}_t \circ \psi^{-1}_{t^K_{k-1}}\d t - \int_{t^K_{k-1}}^{t^K_k}\dot{\psi}_t \circ \psi^{-1}_{t^K_{k-1}}\d t \right) \\
&=K^2 \left(\int_{t^K_{k-1}}^{t^K_k}\int_0^{\frac{1}{K}} \ddot{\psi}_{t+\tau} \circ\psi^{-1}_{t^K_{k-1}} \d \tau \d t \right)\\
&= K^2 \left(\int_{t^K_{k-1}}^{t^K_k}\int_0^{\frac{1}{K}} a_{t + \tau} \circ \psi_{t^K_{k-1},t+\tau} \d \tau \d t \right),
\end{align*}
where in the second equality we used  \eqref{eq:repre_limsup}, and in the last equality \eqref{eq:second_order_flow}. Then, using the Cauchy-Schwarz inequality and \eqref{eq:H^m_composition_estimate} we obtain the following estimate
\begin{align}
\max_k \|a_k^K\|_{C^1(\overline{\Omega})} 
\leq CK^{\frac12}\!\sup_{t\in[0,1],\ 0<\tau\leq K^{-1}}\!\left\vert\int_{t}^{t+2\tau}\!\left\Vert a_s\right\Vert^2_{H^{m}(\Omega)} \d s \right\vert^{\frac12}\!.\!\label{eq:limsup_acceleration_norm_estimate}
\end{align}
The same Taylor expansion argument as in \eqref{eq:taylor_exp1} and \eqref{eq:taylor_exp_r} now implies, together with $\eqref{eq:limsup_acceleration_norm_estimate}$
\begin{align}
&\!\int_\Omega\!\energyDensity_A(Da_k^K)+\gamma\vert D^{m}a_k^K\vert^2\d x \nonumber\\
\!\leq\!& \!\int_\Omega\! L[a_k^K,a_k^K]\d x \!+\!CK^{-\frac32}\sqrt{\mathcal{F}^{\sigma}[u]}. \label{eq:lim_sup_taylor}
\end{align}
%Summing the second term on the right hand side over $k$ and taking \eqref{eq:limsup_acceleration_norm_estimate} into account, we obtain
%\begin{align*}
%\frac{C}{K}\sum_{k=1}^{K-1}\int_\domain \vert Da_k^{K}\vert^3\d x\leq \frac{C}{K}\sum_{k=1}^{K-1}\Vert a_k^K\Vert^3_{C^1(\overline{\domain})}\leq CK^{-\frac{3}{2}}.
%\end{align*}
Applying Jensen's inequality twice on $L$, and taking into account the above \change{relation} between $a_k^K$ and $a$ gives
\begin{align*}
&\int_\Omega L[a_k^K,a_k^K]\d x\\
=&\int_\Omega K^4 L\left[\int_{t^K_{k-1}}^{t^K_k}\int_0^{\frac{1}{K}} a_{t + \tau} \circ \psi_{t^K_{k-1},t+\tau} \d \tau \d t, \int_{t^K_{k-1}}^{t^K_k}\int_0^{\frac{1}{K}} a_{t + \tau} \circ \psi_{t^K_{k-1},t+\tau} \d \tau \d t\right]\d x\\
\leq& \int_\Omega K^{2}\int_{t^K_{k-1}}^{t^K_k}\int_0^{\frac{1}{K}} L\big[a_{t + \tau} \circ \psi_{t^K_{k-1},t+\tau},   a_{t + \tau} \circ \psi_{t^K_{k-1},t+\tau} \big] \d \tau \d t \d x.
\end{align*}
We now estimate the summands of $L$ individually. For the first term we use that $|\tr(AB)|\!\leq\! |\tr(A)|\!+\!|\tr A(B-\Id)|$, \eqref{eq:lim_sup_phi_close_to_id} and transformation rule to get
\begin{align*}
&\int_\Omega\int_{t^K_{k-1}}^{t^K_k}\int_0^{\frac{1}{K}}\tr\left(\varepsilon[a_{t + \tau} \circ \psi_{t^K_{k-1},t+\tau}]^2\right) \d \tau \d t \d x\\
\!\leq&\! \!\int_\Omega\!\int_{t^K_{k-1}}^{t^K_k}\!\int_0^{\frac{1}{K}}\!\!\tr\left((\varepsilon[a_{t+\tau}]\circ \psi_{t^K_{k-1},t+\tau}) ^2\right)
\!+\! \tr\left((\varepsilon[a_{t+\tau}]\circ \psi_{t^K_{k-1},t+\tau}) ^2 (\varepsilon[\psi_{t^K_{k-1},t+\tau}]^2\!-\!\Id)\right) \d \tau \d t \d x \\
\!\leq&\! \!\int_\Omega\!\int_{t^K_{k-1}}^{t^K_k}\!\int_0^{\frac{1}{K}}\!\!\tr\left(\varepsilon[a_{t+\tau}]^2\right) \!+\!CK^{-\frac{1}{2}}\Vert a_{t+\tau}\Vert^2_{H^{m}(\Omega)} \d \tau \d t \d x.
\end{align*}
For the second term we use \eqref{eq:H^m_composition_estimate} and the fact for any $0\leq \tilde{m} \leq m$ and $f \in H^m(\domain,\R^d), g \in H^{\tilde{m}}(\domain,\R^d)$ we have $\|fg\|_{H^{\tilde{m}}(\domain)}\leq C \|f\|_{H^m(\domain)}\|g\|_{H^{\tilde{m}}(\domain)}$ \cite[Lemma 2.3]{InKaTo13} to get
\begin{align*}
\int_\domain \vert D^{m}(a_{t + \tau} \circ \psi_{t^K_{k-1},t+\tau}) \vert \d x
 \leq \int_\domain |D^{m-1}a_{t+\tau}\circ \psi_{t^K_{k-1},t+\tau}|\d x+CK^{-\frac{1}{2}}\Vert a_{t+\tau}\Vert_{H^{m}(\Omega)},
\end{align*}
and \change{iterating} this argument and using the transformation formula we have
\begin{align*}
\int_\Omega\int_{t^K_{k-1}}^{t^K_k}\int_0^{\frac{1}{K}}\vert D^{m}(a_{t + \tau} \circ \psi_{t^K_{k-1},t+\tau}) \vert^2\d \tau \d t \d x
\leq\int_{t^K_{k-1}}^{t^K_k}\int_0^{\frac{1}{K}}\vert a_{t + \tau}\vert^2_{H^{m}(\Omega)}+CK^{-\frac{1}{2}}\Vert a_{t+\tau}\Vert^2_{H^{m}(\Omega)}\d \tau \d t.
\end{align*}
\change{Altogether} with \eqref{eq:limsup_acceleration_norm_estimate} and \eqref{eq:lim_sup_taylor} we have
\begin{align*}
\frac{1}{K}\sum_{k=1}^{K-1}\int_\Omega\energyDensity_A(Da_k^K)+\gamma\vert D^{m}a_k^K\vert^2
\leq & K \int_0^{\frac{1}{K}} \int_0^{1-\frac{1}{K}} \int_\domain L[a_{t + \tau},a_{t+\tau}] + \bigO(K^{-\frac12}) \d t \d \tau \\
\leq & \int_0^{1} \int_\domain L[a_{t},a_{t}] + \bigO(K^{-\frac12}) \d t,
\end{align*}
which together with \eqref{eq:lim_sup_w_bound} gives
\begin{equation*}
\mathcal{F}^K[u^K] \leq \mathcal{F}[u] + \bigO(K^{-\frac12}).
\end{equation*}
%Combining all estimates we obtain then
%\begin{align*}
%&\mathcal{F}^K[u^K]\\
%\leq & \frac{1}{K}\sum_{k=1}^{K-1}\int_\Omega\energyDensity_A(Da_k^K)+\gamma\vert D^{m-1}a_k^K\vert^2+
%\frac{1}{\delta}|w_k^K|^2\d x\\
%%\leq & \int_0^{1} \int_\domain \frac{1}{\delta}w^2(t,x) \d x \d t+\sum_{k=1}^{K-1} K^{3}\int_\Omega L^{m-1}[a_k^K,a_k^K]\d x+\bigO(K^{-\frac{1}{2}})\\
%\leq & \int_0^{1} \int_\domain \frac{1}{\delta}w^2(t,x) \d x +\sum_{k=1}^{K-1} K\int_\Omega\int_{t_{k-1}}^{t_k}\int_0^{\frac{1}{K}}L^{m-1}[a(t+\tau,\cdot),a(t+\tau,\cdot)]\d \tau \d t \d x + \bigO(K^{-\frac{1}{2}})\\
%\leq & \int_0^{1} \int_\domain \frac{1}{\delta}w^2(t,x) \d x +\int_\Omega\int_{0}^{1}L^{m-1}[a(t,\cdot),a(t,\cdot)]\d t \d x + \bigO(K^{-\frac{1}{2}})\\
%= &\mathcal{F}(u)+\bigO(K^{-\frac{1}{2}}),
%\end{align*}
This readily implies the $\limsup$-inequality for the pure spline part of the functional $\mathcal{F}^{\sigma,K}$.

Following analogous steps as above, we can fully repeat the procedure from \cite[Theorem 14]{EfNe19} 
%to bound the energy density terms $\energyDensity_D(D\deformation^K_k)$ and high-order Sobolev seminorm by the elliptic term $L^m$:
%\begin{align*}
%&\int_\Omega\energyDensity_D(D\deformation_k^K)+\gamma\vert D^m\deformation_k^K\vert^2\d x\\
%\leq&\int_\Omega L^m[v_k^K,v_k^K]\d x+C\int_\Omega\Vert Dv_k^K\Vert^3\d x\\
%\leq&\int_\Omega L^m[v,v]\d x+O(K^{-\frac{1}{2}}),
%\end{align*}
%as well as the discrete material derivative term
%\begin{align*}
%&\int_\domain |\image^K_{k} \circ \deformation^K_k - \image^K_{k-1} |^2 \d x \\
%\leq & \frac{1}{K}(1+CK^{-\frac12}) \int_{t^K_{k-1}}^{t^K_k} \int_\domain z^2(s,x) \d x \d s
%\end{align*}
to obtain the estimate for the lower order path energy
\begin{align*}
%\label{eq:path_energy_estimate}
\mathcal{E}^K[u^K]\leq \mathcal{E}[u]+\bigO(K^{-\frac{1}{2}}),
\end{align*}
which finally proves the $\limsup$-inequality.

We are left to show that $\image^K \to \image$ in $L^2((0,1)\change{,\imageSpace})$ as $K \to \infty$. 
\change{To this end we introduce the piecewise constant interpolation
\begin{equation}\label{eq:pw_constant_u}
\bar{u}^K_t \coloneqq
\begin{cases}
u^K_0,& t \in [0,t^K_{\frac12}],\\
u^K_k, & t \in (t^K_{k-\frac12},t^K_{k+\frac12}], ~ k=1,\dots,K-1,\\
u^K_K, & t \in [t^K_{K-\frac12},1].
\end{cases}
\end{equation}}
\change{ Then, for $t \in (t^K_{k-\frac12}, t^K_{k+\frac12}]$ with $k=1,\dots,K-1$ we estimate
\begin{align*}
\|u^K_t - \bar{u}^K_{t}\|^2_\imageSpace 
 \leq& C \left(\|(u^K_{k-1}-u^K_k \circ \phi^K_k)\circ x^K_{t}\|^2_\imageSpace 
  + \|(u^K_{k+1} \circ \phi^K_{k+1} \circ \phi^K_k -2 u^K_k \circ \phi^K_k+ u^K_{k-1})\circ x^K_{t}\|^2_\imageSpace \right. \nonumber \\
& \quad \left. + \|u^K_k \circ \phi^K_k \circ x^K_{t} - u^K_k\|^2_\imageSpace \right) \nonumber \\
% \leq& C \left(\|u^K_{k-1}-u^K_k \circ \phi^K_k\|^2_\imageSpace \|\det Dy^K\|_{L^\infty((0,1)\times \domain)} + \|u^K_{k+1} \circ \phi^K_{k+1} \circ \phi^K_k -2 u^K_k \circ \phi^K_k+ u^K_{k-1})\|^2_\imageSpace \|\det Dy^K\|_{L^\infty((0,1)\times \domain)} \right. \nonumber\\
\leq& C \left(K^{-2}\|\hat{z}^K_k\|^2_{L^2(\domain)} \|\det Dy^K\|_{L^\infty((0,1)\times \domain)}  + K^{-4}\|\hat{w}^K_k\|^2_{L^2(\domain)} \|\det Dy^K\|_{L^\infty((0,1)\times \domain)} \right. \nonumber\\
 & \left. \quad ~ + \|u^K_k \circ \phi^K_k \circ x^K_{t} - u^K_k\|^2_\imageSpace \right) \nonumber\\
\leq & C \left( K^{-1} \|z\|^2_{L^2((0,1)\times \domain)} + CK^{-3} \|w\|^2_{L^2((0,1)\times \domain)} + \|u^K_k \circ \phi^K_k \circ x^K_{k,t} - u^K_k\|^2_\imageSpace \right), %\label{eq:pw_constant_vs_spline1}
\end{align*}
where we used the transformation formula, \eqref{eq:lim_sup_phi_close_to_id}, \eqref{eq:lim_sup_w_bound} and an analogous estimate for the first order material derivative.}
\change{For every $t \in (0,1)$ we can find a sequence $\{k(K)\}_{K \in \N}$ such that for any $K$ large enough $t \in (t^K_{k(K)-\frac12}, t^K_{k(K)+\frac12}]$.  We uniformly approximate the sequence $\{u^K_{k(K)}\}_{K \in \N}$ by smooth functions (cf. \eqref{eq:fixed_u_composition_convergence}) and use \eqref{eq:lim_sup_phi_close_to_id} and \eqref{eq:limsup_acceleration_norm_estimate} to prove
\begin{equation*}
\|u^K_{k(K)} \circ \phi^K_{k(K)} \circ x^K_{t} - u^K_{k(K)}\|^2_\imageSpace \to 0,~ \text{uniformly in } t.
\end{equation*} }
\change{Plugging this back above we have that $u^K - \bar{u}^K \to 0 $ in $L^\infty([0,1],\imageSpace)$ as $K \to \infty$. On the other hand, as $\{\bar{u}^K\}_{K \in \N}$ is a sequence of piecewise constant approximations of $u \in C^1([0,1],\imageSpace)$ we have that $\bar{u}^K \to u$ in $L^2([0,1],\imageSpace)$ as $K \to \infty$. Altogether, $u^K \to u$ in $L^2((0,1),\imageSpace)$ as we wanted to show. This concludes the proof of $\limsup$ estimate and the construction of recovery and thus concludes the proof of Mosco convergence.}

\hspace*{0.1cm}

As a corollary of the previous theorem we are able to show the existence of the continuous time spline defined in Section~\ref{sec:time_continuous}.
To this end, let $J\geq 2$ and $(t_1,\ldots,t_J)\subset[0,1]\cap \mathbb{Q}$ be a sequence of fixed times.
Then for infinitely many  $K\in\mathbb{N}$ one can choose $i_j^K\coloneqq K\cdot t_j\in\mathbb{N}$ for all $j=1,\ldots,J$.
Let $({u}^I_j)_{j=\change{1},\ldots,J}\subseteq\mathcal{I}$ be the set of constraint images at the corresponding constraint times. 

\begin{theorem}[Convergence of discrete spline interpolations to continuous ones]\label{thm:convergence_and_existence}
For every $K$ that satisfies the above condition, let $u^K \in L^2([0,1],\imageSpace)$ be a minimizer of $\mathcal{F}^{\sigma,K}$ among the image curves satisfying $u^K=\mathcal{U}^\change{K}[\imagevec^K,\defvec^K]$
with $u^K_{i_j^K}={u}^I_j$ for all $j \in \{\change{1},\dots,J\}$.
Then a subsequence of $\lbrace u^K\rbrace_{K\in\N}$ converges weakly in $L^2([0,1],\imageSpace)$ to a minimizer of the continuous spline energy $\mathcal{F}^\sigma$ as $K\rightarrow\infty$.
This minimizer satisfies $u_\change{{t_j}}={u}^I_j$  for all $j \in \{\change{1},\dots,J\}$ and the associated sequence of discrete energies converges to the minimal continuous spline energy. 
\end{theorem}
\begin{proof}
For $j=1,\dots,J$ let $\eta^j:[0,1]\to \R$ be smooth functions with $\eta^j(t_i)=\delta_{ij}$.
We define a smooth interpolating curve of the fixed images $\tilde{u}(t)\coloneqq\sum_{j=0}^J\eta^{j}(t){u}^I_j$.
%For $a\in[0,1]$, $\epsilon>0$, let $\eta^{a,\epsilon}\in C^\infty_c([0,1])$ be a smooth function centred at $a$, with $\text{supp}(\eta^{a,\epsilon})\subseteq[a-\epsilon,a+\epsilon]\cap[0,1]$, and $\eta^{a,\epsilon}(a)=1$. Define $\epsilon<\frac{1}{2}\min_j\lbrace t_{j+1}-t_j\rbrace$, $\tilde{u}(t)\coloneqq\sum_{j=0}^J\eta^{t_j,\epsilon}(t){u}^I_j$. Then, the image curve $\tilde{u}\in C^\infty([0,1],\imageSpace)$ is a smooth interpolating curve of the fixed images.
Let \change{us define the vector} $\tilde{\imagevec}^K\coloneqq(\tilde{u}(t^K_\change{1}),\ldots, \tilde{u}(t_K^K))$ and \change{its time extension} $\tilde{u}^K\coloneqq\mathcal{U}^\change{K}[\tilde{\imagevec}^K,\Id^K]$.
This image curve gives an admissible candidate for a minimizer of the functional $\mathcal{F}^{\sigma,K}$.
Namely,
\begin{align*}
\splinepathenergy^{\sigma,K}[\tilde{\image}^K]\leq& \mathbf{F}^{\sigma,K,D}[\tilde{\imagevec}^K,\Id^K]\\
=&\sigma K\sum_{k=0}^K \int_\domain|\change{\tilde{\image}^K_{k+1}-\tilde{\image}^K_k}|^2 \d x + K^3 \sum_{k=1}^{K-1} \int_\domain |\change{\tilde{\image}^K_{k+1}-2\tilde{\image}^K_k + \tilde{\image}^K_{k-1}}|^2 \d x \\
\leq & C\left(\int_{\domain} |\tilde{\image}|^2_{H^1((0,1))}+|\tilde{\image}|^2_{H^2((0,1))} \d x +1\right),
%\leq&C(M_1+M_2+1)\coloneqq \overline{\mathcal{F}}
\end{align*}
%where $M_1\coloneqq\Vert\max_j\lbrace\vert {u}^I_j\vert\rbrace\Vert_{L^2(\domain)}^2\vert\sup_{t\in[0,1]}\dot{\eta}^{a,\epsilon}(t)\vert^2$, and $M_2\coloneqq\Vert\max_j\lbrace\vert {u}^I_j\vert\rbrace\Vert_{L^2(\domain)}^2\vert\sup_{t\in[0,1]}\ddot{\eta}^{a,\epsilon}(t)\vert^2$,
%both of which are finite and independent of $K$.
where the upper bound is independent of $K$.
As defined above, let $\{\image^K\coloneqq\mathcal{U}^K[\imagevec^K,\defvec^K]\}_{K \in \N}$, where $\imagevec^K,\defvec^K$ are  optimal pairs for the discrete spline (see Theorem \ref{thm:discrete_spline_existence}). In particular, we have  $\mathcal{F}^{\sigma,K}[\image^K]=\mathbf{F}^{\sigma, K}[\imagevec^K,\defvec^K]<\overline{\mathcal{F}}$.
%The spline energy of $\image^K$ can be bounded by 
%$$\splinepathenergy^{\sigma,K}[\image^K]\leq\splinepathenergy^{\sigma,K}[\tilde\image^K]=C(M_1+M_2)\eqqcolon\overline{\splinepathenergy}.$$
%For optimal vectors of images $\mathbf{u}^K$ and deformations $\defvec^K$ in the definition of $\mathcal{F}^{\sigma,K}$, we apply the temporal extension construction from Section~\ref{sec:time_extension}. In particular, it holds $\mathbf{F}^{\sigma, K}[\mathbf{u}^K,\mathbf{\Phi}^K]\leq\overline{\mathcal{F}}$ for all $K \in \N$. 
%Following the same steps as in the proof of the $\liminf$-inequality, we can conclude that the discrete transport path converges uniformly to the identity mapping and that $z^K$ and $w^K$ are uniformly bounded in $L^2((0,1)\times\Omega)$ and the uniform boundedness of $\psi^K$ in $C^1([0,1],C^{1,\alpha}(\overline{\Omega}))$.
As in \eqref{eq:uniformBoundImagesInduction} we show uniform boundedness of $u^K_k$ in $L^2(\Omega)$. This further implies, together with boundedness of the deformations and the convergence of the discrete transport paths to the identity that $u^K_t$ is uniformly bounded in $L^2(\domain)$.
Therefore, $\lbrace u^K\rbrace_{K\in\N}$ is uniformly bounded in $L^\infty([0,1],\imageSpace)$ and a subsequence converges weakly to some $u\in L^2([0,1],\imageSpace)$. 

Now, we follow the usual argument and assume that there exists an image path $\hat{u}\in L^2([0,1],\mathcal{I})$ with finite energy
%with the corresponding optimal tuple $(\hat{v},\hat{a},\hat{z},\hat{w})$, which exists due to Proposition~\ref{prop:optimal_tuple},
such that $\mathcal{F}^\sigma[\hat{u}]<\mathcal{F}^\sigma[u].$
By the $\limsup$-part of Theorem~\ref{thm:mosco_convergence} there exists a sequence $\lbrace\hat{u}^K\rbrace_{K\in\N}\subseteq L^2((0,1),\mathcal{I})$ such that  $\limsup_{K\rightarrow\infty} \mathcal{F}^{\sigma,K}[\hat{u}^K]\leq\mathcal{F}^\sigma[\hat{u}]$. Now, we apply the $\liminf$-part of Theorem~\ref{thm:mosco_convergence}, thus obtaining
\begin{align}
\mathcal{F}^{\sigma}[u]\leq& \liminf_{K\rightarrow\infty}\mathcal{F}^{\sigma, K}[u^K]
\leq{\liminf}_{K\rightarrow\infty}\mathcal{F}^{\sigma, K}[\hat{u}^K]\leq\mathcal{F}^\sigma[\hat{u}],\label{eq:thm2_2}
\end{align}
which is a contradiction to the above assumption. Hence, $u$ minimizes the continuous spline energy over all admissible image curves and discrete spline energies converge to the limiting spline energy along a subsequence, i.e. $\lim_{K\rightarrow\infty}\mathcal{F}^{\sigma,K}[u^K]=\mathcal{F}^\sigma[u]$, which follows from \eqref{eq:thm2_2} by using $\hat{u}=u$.

Finally, we show that $u_{t_j}=u^I_j$. 
%To this end we introduce the piecewise constant interpolation
%\begin{equation*}
%\bar{u}^K_t=
%\begin{cases}
%u^K_0,& t \in [0,t^K_{\frac12})\\
%u^K_k, & t \in [t^K_{k-\frac12},t^K_{k+\frac12}), ~ k=1,\dots,K-1\\
%u^K_k, & t \in [t^K_{K-\frac12},1],
%\end{cases}
%\end{equation*}
%and \change{straightforwardly} check that $u^K_t-\bar{u}^K_t \to 0$ in $\imageSpace$, uniformly in $t$.
\change{To this end recall that for the sequence of  piecewise constant approximations $\{\bar{u}^K\}_{K \in \N}$ given by \eqref{eq:pw_constant_u} we showed that $u^K_t-\bar{u}^K_t \to 0$ in $\imageSpace$, uniformly in $t$.} 
Together with $u^K_t \rightharpoonup u_t$ in $\imageSpace$ for every $t \in [0,1]$ which is showed by a trace theorem type argument (see \cite[Theorem 4.1, (iv)]{BeEf14}) we have the needed result since $\bar{u}^K_{t_j}={u}^I_j$.
\end{proof}
The analogous result for arbitrary $(t_1,\ldots,t_J)\subset[0,1]$ follows from density of $\mathbb{Q}$ in $[0,1]$.
Let us remark that in light of Proposition~\ref{prop:equivalence_of_approaches} \change{we have that} for optimal scalar quantities $z,w$ it holds $z=|\hat{z}|$ and $w=|\hat{w}|$.

%%%%%%%%%%%%%%%%%%%%%%%%%%%%%%%%%%%%%%%%%%%%%%%%%%%%%%%%%%%%%%%%%%%%%%%%%

\section{Fully Discrete Metamorphosis Splines}\label{sec:fully_discrete}
To \change{numerically} implement splines for image metamorphosis we have to further discretize the space.
Here, we present a model for $c$ image channels and a two dimensional image domain $\Omega\coloneqq[0,1]^2$ \change{and follow the fully discrete version of image metamorphosis introduced in~\cite{EfKoPo19a}.} 
\change{Before presenting the details of the space discretization,} to avoid double warping in the second material derivative term \eqref{eq:discrete_w_definition} and to further increase the robustness of the model  
we explicitly introduce a vector valued material derivative $\bar z \in L^2((0,1),L^2(\domain,\R^\change{c}))$\change{. This leads to} a relaxation of \eqref{eq:continuous_spline_energy}:
\change{
\begin{align*}
&\mathcal{F}^\sigma[u]
\coloneqq\inf_{(v,a,\change{\hat{z},\bar{z}},w)}\int_0^1 \int_{\Omega} L[a,a] + \frac{1}{\delta} |w|^2 +\sigma(L[v,v]+\frac{1}{\delta}|\bar{z}|^2)+\frac{1}{\theta}|\bar{z}-\hat z|^2 \d x \d t,
\end{align*}
}
with a penalty on the misfit of the new variable $\bar{z}$ and the actual material derivative $\hat z$ \change{given by} \eqref{eq:first_variational_equality}, 
while $w$ is the  \change{material derivative} of $\bar{z}$.
\change{The time discrete counterpart $\SplinePathenergy^{\sigma,K}[\imagevec]$ is defined by 
\begin{equation*}
 \SplinePathenergy^{\sigma,K}[\imagevec] \coloneqq \inf_{\bar{\mathbf{z}} \in L^2(\domain,\R^c)^K, \defvec \in \admset^K} \SplinePathenergy^{\sigma,K,D}[\imagevec,\bar{\mathbf{z}},\defvec],
\end{equation*}
where, for}
$\bar{\mathbf{z}} = (\bar{{z}}_1,\ldots, \bar{{z}}_K)$ we write
\begin{align*}
\SplinePathenergy^{\sigma,K,D}[\imagevec,\bar{\mathbf{z}},\defvec] 
\!\change{\coloneqq}\!&\int_{\domain} \sum_{k=1}^{K-1}\! \! \tfrac{1}{K}\left(\energyDensity_A(Da_k)\!+\!\gamma |D^{m} a_k|^2\right)
\!+\!\tfrac{K}{\delta}|\bar{z}_{k+1} \circ \deformation_k - \bar{z}_{k}|^2\\
&\!+\!\sigma\left(\sum_{k=1}^K K \energyDensity_D(D\deformation_k)+K\gamma |D^{m} \deformation_k|^2+\tfrac{1}{\delta K}|\bar{z}_{k}|^2 \right)
\!+\!\sum_{k=1}^K\tfrac{1}{\theta K}\left|K(\image_k \circ \deformation_k - \image_{k-1}) - \bar{z}_{k}\right|^2 \d x\change{.}
\end{align*}
\change{Here,} $\bar{{w}}_k=K(\bar{z}_{k+1} \circ \deformation_k - \bar{z}_{k})$ is the discrete material derivative of $\bar{z}_{k}$,
while $\hat{z}_k=K(\image_k \circ \deformation_k - \image_{k-1})$ is the actual material derivative of $u_k$, with the corresponding second order derivative $(\hat{w}_k)_{k=1}^{K-1}$ given by \eqref{eq:discrete_w_definition}.

For $M,N \geq 3$ we define the computational domain
\[
\discreteDomain\!\change{\coloneqq}\!\left\{\tfrac{0}{M-1},\tfrac{1}{M-1},\ldots,\tfrac{M-1}{M-1}\right\}\!\times\!\left\{\tfrac{0}{N-1},\tfrac{1}{N-1},\ldots,\tfrac{N-1}{N-1}\right\},
\]
with discrete boundary 
$\partial\discreteDomain\change{\coloneqq}\discreteDomain \cap \partial([0,1]^2)
%\backslash\{\tfrac{1}{M-1},\ldots,\tfrac{M-2}{M-1}\}\times\{\tfrac{1}{N-1},\ldots,%\tfrac{N-2}{N-1}\}
$ 
and $\left\|\discreteImage\right\|_{L^p_{\MN}}^p\change{\coloneqq}\frac{1}{\MN}\sum_{(\change{\discretex,\discretey})\in\discreteDomain}\|\discreteImage(\change{\discretex,\discretey})\|_p^p$.
The discrete image space is $\imageSpace_{\MN}\coloneqq\{\discreteImage:\discreteDomain\to\R^\change{c}\}$ and
the set of admissible deformations is
\begin{align*}
\mathcal{D}_{\MN}\!\change{\coloneqq}\!&\left\{\discreteDeformation=\change{(\discreteDeformation^1,\discreteDeformation^2)}\!:\!\discreteDomain\!\to\![0,1]^2, \discreteDeformation\!=\!\Id\!\text{ on }\!\partial\discreteDomain,\ \det(\nabla_{\MN}\discreteDeformation)\!>\!0\right\},
\end{align*}
where the discrete Jacobian operator of $\discreteDeformation$ at $\change{(\discretex,\discretey)}\in\discreteDomain$ is defined  
as the forward finite difference operator with Neumann boundary conditions.
Here and in the rest of the paper we used bold faced letters for fully discrete quantities.
A spatial warping operator~$\warp$ that approximates
the pullback of an image channel~$\discreteImage^j\circ\discreteDeformation$ at a point $\change{(\discretex,\discretey)}\in\discreteDomain$ \change{is defined} by
\begin{align*}
&\warp[\discreteImage^j,\discreteDeformation]\change{(\discretex,\discretey)}
\change{\coloneqq}\sum_{\change{(\tilde\discretex,\tilde\discretey)}\in\discreteDomain}s(\discreteDeformation^\change{1}\change{(\discretex,\discretey)}
-\change{\tilde\discretex})s(\discreteDeformation^{\change{2}}\change{(\discretex,\discretey)}-\change{\tilde\discretey})\discreteImage^j\change{(\tilde\discretex,\tilde\discretey)}\,,
\end{align*}
where $s$ is the third order B-spline interpolation kernel.
This form of warping is also used for composition of deformations, i.e we define the fully discrete acceleration as an approximation of \eqref{eq:discrete_acceleration_definition} by
\begin{equation*}
\discreteAcceleration^j_k\coloneqq K^2(\warp[\discreteDeformation^j_{k+1}-\Id,\discreteDeformation_k]-(\discreteDeformation^j_k -\Id)),~ j=1,2.
\end{equation*}
In summary, the fully discrete spline energy in the metamorphosis model for a 
$(K+1)$-tuple $(\discreteImage_k)_{k=0}^K$ of discrete images\change{,} a $K$-tuple $(\bar{\discreteDerivative}_k)_{k=1}^{K}$ of discrete derivatives \change{and a $K$-tuple $(\discreteDeformation_k)_{k=1}^K$ of discrete deformations}
reads as
\begin{align*}
\change{\SplinePathenergy_{\MN}^{\sigma,K}[(\discreteImage_k)_{k=0}^K]}
\change{\coloneqq}&\inf\limits_{\change{\bar{\discreteDerivative} \in \imageSpace_{\MN}^K,} \defvec \in \mathcal{D}_{\MN}^K}\SplinePathenergy_{\MN}^{\sigma,K,D}[(\discreteImage_k)_{k=0}^K,(\bar\discreteDerivative_k)_{k=1}^K,(\discreteDeformation_k)_{k=1}^K]\\
=&\!\inf\limits_{\change{\bar{\discreteDerivative} \in \imageSpace_{\MN}^K,}\defvec \in \mathcal{D}_{\MN}^K}\!\sum_{k=1}^{K-1}\! \tfrac{1}{K} \Vert\energyDensity_A(\nabla_{\MN}\discreteAcceleration_k)\Vert_{L^1_{\MN}} \!+\!\tfrac{K}{\delta} \DataEnergy^s_{\MN}[\bar\discreteDerivative_{k},\bar\discreteDerivative_{k+1},\discreteDeformation_k]\\
&\quad \quad \quad \quad  \!+\!\sum_{k=1}^K \sigma \left(K\Vert\energyDensity_D(\nabla_{MN}\discreteDeformation_k)\Vert_{L^1_{\MN}}\!+\! \tfrac{1}{\delta K}  \Vert\bar\discreteDerivative_{k}\Vert^2_{L^2_{\MN}}\right) + \tfrac{1}{\theta K} \DataEnergy^g_{\MN}[\discreteImage_{k-1},\discreteImage_k,\bar\discreteDerivative_{k},\discreteDeformation_k]\,,
\end{align*}
where
\begin{align*}
\DataEnergy^s_{\MN}[\discreteDerivative,\tilde\discreteDerivative,\discreteDeformation]
\change{\coloneqq}&\frac{1}{2c}\sum_{j=1}^{c}\left\Vert\warp[\tilde\discreteDerivative^j,\discreteDeformation]-\discreteDerivative^j\right\Vert_{L^2_{\MN}}^2\change{,}\\
\DataEnergy^g_{\MN}[\discreteImage,\tilde\discreteImage,\discreteDerivative,\discreteDeformation]
\change{\coloneqq}&\frac{1}{2c}\sum_{j=1}^{c}\left\Vert K (\warp[\tilde\discreteImage^j,\discreteDeformation]-\discreteImage^j) - \discreteDerivative^j\right\Vert_{L^2_{\MN}}^2\,.
\end{align*}
%For simplicity, we neglect the higher order Sobolev norm terms in this fully discrete model.
While in the spatially continuous context the compactness induced by the 
$H^{m}$-seminorm is indispensable, in this fully discrete model grid dependent regularity is ensured by the use of cubic B-splines.
Thus, we dropped the higher order Sobolev norm terms in this fully discrete model.

To improve the robustness of the overall optimization, we take into account a multiresolution strategy. In detail, on the coarse computational domain of size $M_{L}\times N_{L}$ with $M_{L}=2^{-(L-1)}M$ and $N_{L}=2^{-(L-1)}N$ for a given $L\geq1$,
a time discrete spline sequence $\change{(\discreteImage_k)_{k=0}^{K}}$ is computed as minimizer of $\SplinePathenergy_{M_{L} N_{L}}^{\sigma,K}$ subject to given fixed images $\discreteImage_{i_j}=\discreteImage^I_{j}, ~ j=1,\dots, J$.
In subsequent prolongation steps, the width and the height of the computational domain are successively doubled and the initial deformations, images and derivatives are obtained
via a bilinear interpolation of the preceding coarse scale solutions.
%%%%%%%%%%%%%%%%%%%%%%%%%%%%%%%%%%%%%%%%%%%%%%%%%%%%%%%%%%%%%

\section{Numerical Optimization Using the iPALM Algorithm}\label{sec:algorithm}
In this section, we discuss the numerical solution of the above fully discrete variational problem
based on the application of a variant of the inertial proximal alternating linearized minimization algorithm 
(iPALM, \cite{PoSa16}). \change{Using this algorithm effective optimization results were achieved for a wide range of non convex and non smooth problems. In particular, it was already used for numerical optimization in the context of  the deep feature metamorphosis model~\cite{EfKoPo19a}.}
Following \cite{EfKoPo19a}, to enhance the stability the warping operation is linearized with respect to the deformation at 
$\discreteDeformation^{[\beta]}\in\deformationSpace_{\MN}$ \change{coming} from the previous iteration which leads to 
the modified energies
\begin{align*}
\tilde{\DataEnergy}^s_{\MN}[\discreteDerivative,\tilde\discreteDerivative,\discreteDeformation,\discreteDeformation^{[\beta]}]
\change{\coloneqq}&\tfrac{1}{2c}\!\sum_{j=1}^{c}\!\left\Vert\warp[\tilde\discreteDerivative^j,\discreteDeformation^{[\beta]}]\!+\!\left\langle\Lambda_j(\discreteDerivative,\tilde\discreteDerivative,\discreteDeformation^{[\beta]}),\discreteDeformation\!-\!\discreteDeformation^{[\beta]}\right\rangle\!-\!\discreteDerivative^j\right\Vert_{L^2_{\MN}}^2\\
\tilde{\DataEnergy}^g_{\MN}[\discreteImage,\tilde\discreteImage,\discreteDerivative,\discreteDeformation,\discreteDeformation^{[\beta]}]
\change{\coloneqq}&\tfrac{1}{2c}\!\sum_{j=1}^{c}\!\left\Vert K \warp[\tilde\discreteImage^j,\discreteDeformation^{[\beta]}]
\!+\!\left\langle\Lambda_j(K\discreteImage \!+\! \discreteDerivative,K\tilde\discreteImage,\discreteDeformation^{[\beta]}),\discreteDeformation\!-\!\discreteDeformation^{[\beta]}\right\rangle-(K\discreteImage^\change{j}+ \discreteDerivative^j)\right\Vert_{L^2_{\MN}}^2\,,
\end{align*}
with 
\begin{equation*}
\Lambda_j(\discreteImage,\tilde\discreteImage,\discreteDeformation^{[\beta]})=\tfrac{1}{2}(\nabla_\change{\MN}\warp[\tilde\discreteImage^j,\discreteDeformation^{[\beta]}]+\nabla\discreteImage^j).
\end{equation*}
To further stabilize the computation, the Jacobian operator
applied to the images is approximated using a Sobel filter.
Here, $\langle \cdot, \cdot \rangle$ represent the pointwise product of the involved matrices.
We use the proximal mapping of a functional~$f:\deformationSpace_{\MN} \to(-\infty,\infty]$ for $\tau>0$ given as
\begin{equation*}
\operatorname{prox}_\tau^f\change{[\discreteDeformation]}\coloneqq\argmin_{\change{\tilde\discreteDeformation} \in \deformationSpace_{\MN}}\left(\frac{\tau}{2}\left\|\change{\discreteDeformation-\tilde\discreteDeformation}\right\|^2_\change{{L^2_{\MN}}}
+f(\change{\tilde\discreteDeformation})\right).
\end{equation*}
Then, with the function values on $\partial\discreteDomain$ remaining unchanged, the proximal operator we are interested in is given by
\begin{align*}
&\operatorname{prox}_{\tau}^{\frac{K}{\delta}\widetilde\DataEnergy^s_{\MN}+\frac{1}{K\theta}\widetilde\DataEnergy^g_{\MN}}[\discreteDeformation^{t}_k]\\
=&\left(\Id+\tfrac{K}{c\tau\delta}\sum_{j=1}^c|\Lambda^s_j|^2 + \tfrac{1}{c\tau\theta K}\sum_{j=1}^c |\Lambda^g_j|^2\right)\\
\!&\!\Big(\!\discreteDeformation_k^{t}-\tfrac{K}{c\tau\delta}\sum_{j=1}^c\Lambda^s_j\big(\warp[\bar\discreteDerivative_{k+1}^j,\discreteDeformation^{[\beta]}_k]
-(\Lambda^s_j)^T\discreteDeformation^{[\beta]}_k-\bar\discreteDerivative_{k}^j\big) \!-\! \tfrac{1}{c\tau\theta K}\!\sum_{j=1}^c\! \Lambda^g_j \big(\warp[\!K\!\discreteImage_k^j,\discreteDeformation_k^{[\beta]}]\!-\!(\Lambda^g_j)^T\discreteDeformation_k^{[\beta]}\!-\!K\!\discreteImage_{k-1}^j\!-\!\bar\discreteDerivative_{k}^j\big) \Big),
\end{align*}
where $\Lambda^s_j\coloneqq\Lambda(\bar\discreteDerivative_{k}^j,\bar\discreteDerivative_{k+1}^j,\discreteDeformation_k^{[\beta]})$ and $\Lambda^g_j\coloneqq\Lambda(K\discreteImage_{k-1}^j+\bar\discreteDerivative_{k}^j,K\discreteImage_k^j,\discreteDeformation_k^{[\beta]})$. The first terms in both brackets are activated only for $k<K$.

%In the $l^{th}$ iteration of the algorithm for the minimization of the spline energy~$\SplinePathenergy_{\MN}^{\sigma,K}$ 
%the $k^{th}$ path elements are updated as follows
%\begin{align*}
%\discreteDeformation_k^{l,t}&=\discreteDeformation_k^{[\beta,l]}-\tfrac{1}{L[\discreteDeformation_k^{[l]}]}\nabla_{\discreteDeformation_k}\Bigg(\sigma K\Vert\energyDensity_D(\nabla_{\MN}{\discreteDeformation}_k^{[\beta,l]})\Vert_{L^1_{\MN}} \\
%&\hspace*{0.5cm}+ \frac{1}{K}\Vert \energyDensity_A (\nabla_{\MN}{\discreteAcceleration}^{[\beta,l]}_k)\!+\!\energyDensity_A(\nabla_{\MN}\discreteAcceleration^{[\beta,l]}_{k-1}) \Vert_{L^1_{\MN}} \Bigg),\\
%\discreteDeformation_k^{[l+1]}&=\operatorname{prox}_{L[\discreteDeformation_k^{[l]}]}^{\frac{K}{\delta}\widetilde\DataEnergy^s_{\MN}+\frac{1}{K\theta}\widetilde\DataEnergy^g_{\MN}}[\discreteDeformation_k^{l,t}],\\
%\bar{\discreteDerivative}_k^{[l+1]}&={\bar\discreteDerivative^{[\beta,l]}}_k-\dfrac{\nabla_{\bar\discreteDerivative_k}\SplinePathenergy_{\MN}^{\sigma,K,D}[\discreteImage^{[k,l]},\bar\discreteDerivative^{[\beta,k,l]},
%\discreteDeformation^{[k+1,l]}]}{L[\bar\discreteDerivative_k^{[l]}]},\\
%\discreteImage_k^{[l+1]}&={\discreteImage}_k^{[\beta,l]}\!-\!\dfrac{\nabla_{\discreteImage_k}\SplinePathenergy_{\MN}^{\sigma,K,D}[{\discreteImage}^{[\beta,k,l]},\bar\discreteDerivative^{[k+1,l]},\discreteDeformation^{[k+1,l]}]}{L[\discreteImage_k^{[l]}]}.
%\end{align*}

\change{
\begin{algorithm}
\change{
\SetInd{1ex}{1ex}
\For{$i=1$ \KwTo $I$}{
\For{$k=1$ \KwTo $K$}{
\tcc{update deformation}
$\discreteDeformation_k^{i,t}=\discreteDeformation_k^{[\beta,i]}-\tfrac{1}{L[\discreteDeformation_k^{[i]}]}\nabla_{\discreteDeformation_k}\Bigg(\sigma K\Vert\energyDensity_D(\nabla_{\MN}{\discreteDeformation}_k^{[\beta,i]})\Vert_{L^1_{\MN}}+ \frac{1}{K} \Vert \energyDensity_A (\nabla_{\MN}{\discreteAcceleration}^{[\beta,i]}_k)\!+\!\energyDensity_A(\nabla_{\MN}\discreteAcceleration^{[\beta,i]}_{k-1}) \Vert_{L^1_{\MN}} \Bigg)$\;
$\discreteDeformation_k^{[i+1]}=\operatorname{prox}_{L[\discreteDeformation_k^{[i]}]}^{\frac{K}{\delta}\widetilde\DataEnergy^s_{\MN}+\frac{1}{K\theta}\widetilde\DataEnergy^g_{\MN}}[\discreteDeformation_k^{i,t}]$\;
\tcc{update derivative}
$\bar{\discreteDerivative}_k^{[i+1]}={\bar\discreteDerivative^{[\beta,i]}}_k-\dfrac{\nabla_{\bar\discreteDerivative_k}\SplinePathenergy_{\MN}^{\sigma,K,D}[\discreteImage^{[k,i]},\bar\discreteDerivative^{[\beta,k,i]},
\discreteDeformation^{[k+1,l]}]}{L[\bar\discreteDerivative_k^{[l]}]}$\;
\If{$k \notin I^K$}{
\tcc{update image}
$\discreteImage_k^{[i+1]}={\discreteImage}_k^{[\beta,i]}\!-\!\dfrac{\nabla_{\discreteImage_k}\SplinePathenergy_{\MN}^{\sigma,K,D}[{\discreteImage}^{[\beta,k,i]},\bar\discreteDerivative^{[k+1,i]},\discreteDeformation^{[k+1,i]}]}{L[\discreteImage_k^{[i]}]}$\;
}
}
}
\caption{Algorithm for minimizing $\SplinePathenergy_{\MN}^{\sigma,K}$.}
\label{algo:optimization}
}
\end{algorithm}

The actual minimization  $\SplinePathenergy_{\MN}^{\sigma,K}$ is performed by   Algorithm~\ref{algo:optimization}.}
We used the following notation for the extrapolation with $\beta>0$ of the $k^{th}$~path element in the \change{$i^{th}$} iteration step 
\begin{align*}
&h^{[\beta,\change{i}]}_k=h_k^{[\change{i}]}+\beta(h_k^{[\change{i}]}-h_k^{[\change{i}-1]}),\\
&h^{[k,\change{i}]} = (h_{1,\ldots, k-1}^{[\change{i}+1]},h_{k,\ldots, K}^{[\change{i}]}),\\
&h^{[\beta,k,\change{i}]}=(h_1^{[\change{i}+1]},\ldots,h_{k-1}^{[\change{i}+1]},h^{[\beta,\change{i}]}_k,h_{k+1}^{[\change{i}]},\ldots,h_{K}^{[\change{i}]}),
\end{align*}
while the acceleration $\discreteAcceleration^{[\beta,\change{i}]}$ is 
computed with correspondingly updated $\discreteDeformation^{[\beta,\change{i}]}$ values. Furthermore, \change{$I^K$ is the set of fixed indices (cf. \eqref{eq:admissible_images}) and} we denote by $L[h]$ the Lipschitz constant of the gradient of the function $h$, which is determined by backtracking~\change{\cite{BeTe09}. 
The discrete deformations are initialized by the identity deformation, the discrete images by piecewise linear interpolation of the key frame images, while the discrete derivatives are initialized as differences of two consecutive images in the sequence.}
%%%%%%%%%%%%%%%%%%%%%%%%%%%%%%%%%%%%%%%%%%%%%%%%%%%%%%%%
\section{Applications}\label{sec:results}
In what follows, we investigate and discuss qualitative properties of the spline interpolation in the space of images, 
being aware that the superior temporal smoothness of this interpolation is difficult to show with series of still images.
For all the examples we use $\change{L=}5$ levels in the multi-\change{level} approach and \change{$I=250$ iterations of iPALM algorithm on each level with  the extrapolation parameter}  $\beta=\frac{1}{\sqrt{2}}$. For the first \change{two examples} $\change{M=}N=64$, while for the \change{others} $\change{M=}N=128$. \change{Also, for the first two and the final example $K=8$, while for the others $K=16$.} For plotting of images we crop the values to $[0,1]$, for the material derivative the values are scaled to the interval $[0,1]$ for plotting, while for the displacement and acceleration plots hue refers to the direction and the intensity is proportional to its norm as indicated by the color wheel.

Figure~\ref{fig:gaussians} shows a first test case. As key frames we consider three images showing 
two dimensional \change{Gaussian} distribution with small variance at different positions and of different mass.
Spline interpolation is compared with piecewise geodesic interpolation. Furthermore, it is depicted that for the metamorphosis spline, the curve in $(x,y,m)$-space (position, mass) corresponds almost perfectly to the cubic spline 
interpolation of the parameters of the Gaussian distribution on the key frames.

\begin{figure*}[htb]
\includegraphics[width=\textwidth]{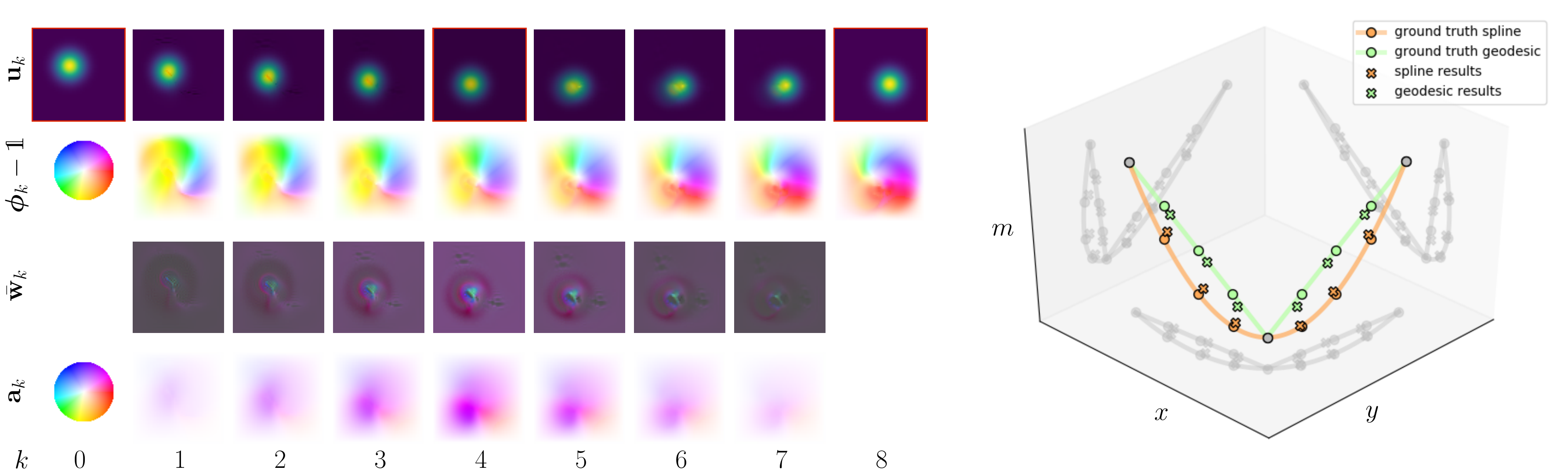}
\caption{Left: Time discrete spline with framed key frame images (first row), color-coded displacement field (second row),  discrete second order material derivative (third row)  and color-coded discrete acceleration field (fourth row) for the \change{Gaussians} example and values of the parameters $\delta=5\cdot 10^{-3}$, $\sigma=1$, $\theta=5\cdot 10^{-5}$. The colors and their intensities indicate the direction and the intensity of the field, as indicated by the color wheel on the left.
Right: Euclidean splines in $(x,y,m)$ coordinates for the input parameters versus  splines for metamorphosis 
in $(x,y,m)$  extracted from the numerical results in post-processing, with $(x,y)$ denoting the center of mass and $m$ the mass of the distribution. }
\label{fig:gaussians}
\end{figure*}
As a next step, we conceptually compare splines for metamorphosis and piecewise geodesic paths in Figure~\ref{fig:circle_square}. For this specific example, we considered the image of a circle and two identical squares as key frames. The influence of the circle's curvature on the spline segment between the two identical squares is visible via concave 'edges', while for the piecewise geodesic interpolation any memory of the circle is naturally lost in the constant interpolation between the two squares.
\begin{figure*}[!h]
\includegraphics[width=\textwidth]{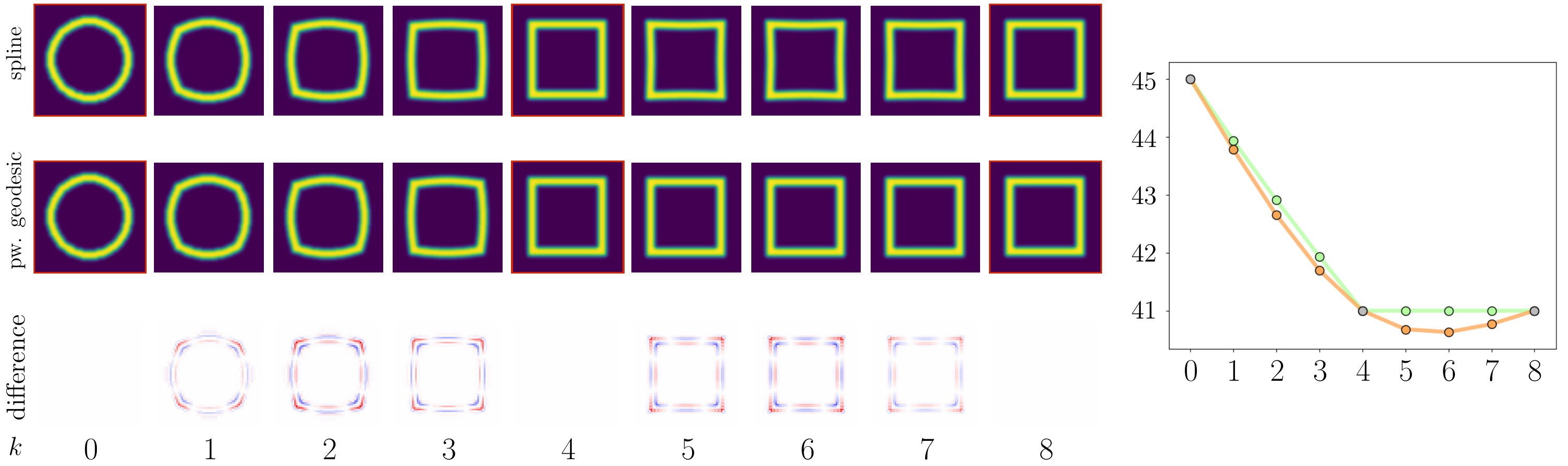}
\caption{Left: Time discrete spline (top row) and piecewise geodesic (\change{middle row}) interpolation with framed key frames. \change{The bottom row shows the difference in intensity between the different interpolations, using the color map $-0.35\hspace{1mm}$\protect\resizebox{.08\linewidth}{!}{\protect\includegraphics{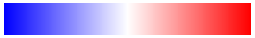}}$\hspace{1mm}0.35$.}
Right: Width of the interpolated shape measured at the horizontal axis of symmetry (in number of pixels) for a spline interpolation (orange) and piecewise geodesic interpolation (green) showing the concavities ($\delta=5\cdot 10^{-3}$, $\sigma=1$, $\theta=5\cdot 10^{-4}$). }
\label{fig:circle_square}
\end{figure*}

Next, we consider spline interpolation between three human portraits on Figure \ref{fig:human_faces}. 
The plots of second material derivative and acceleration show a strong concentration around the key frames, where
the spline is expected to be smooth and the piecewise geodesic path just Lipschitz.
\begin{figure*}[htb]
\includegraphics[width=\textwidth]{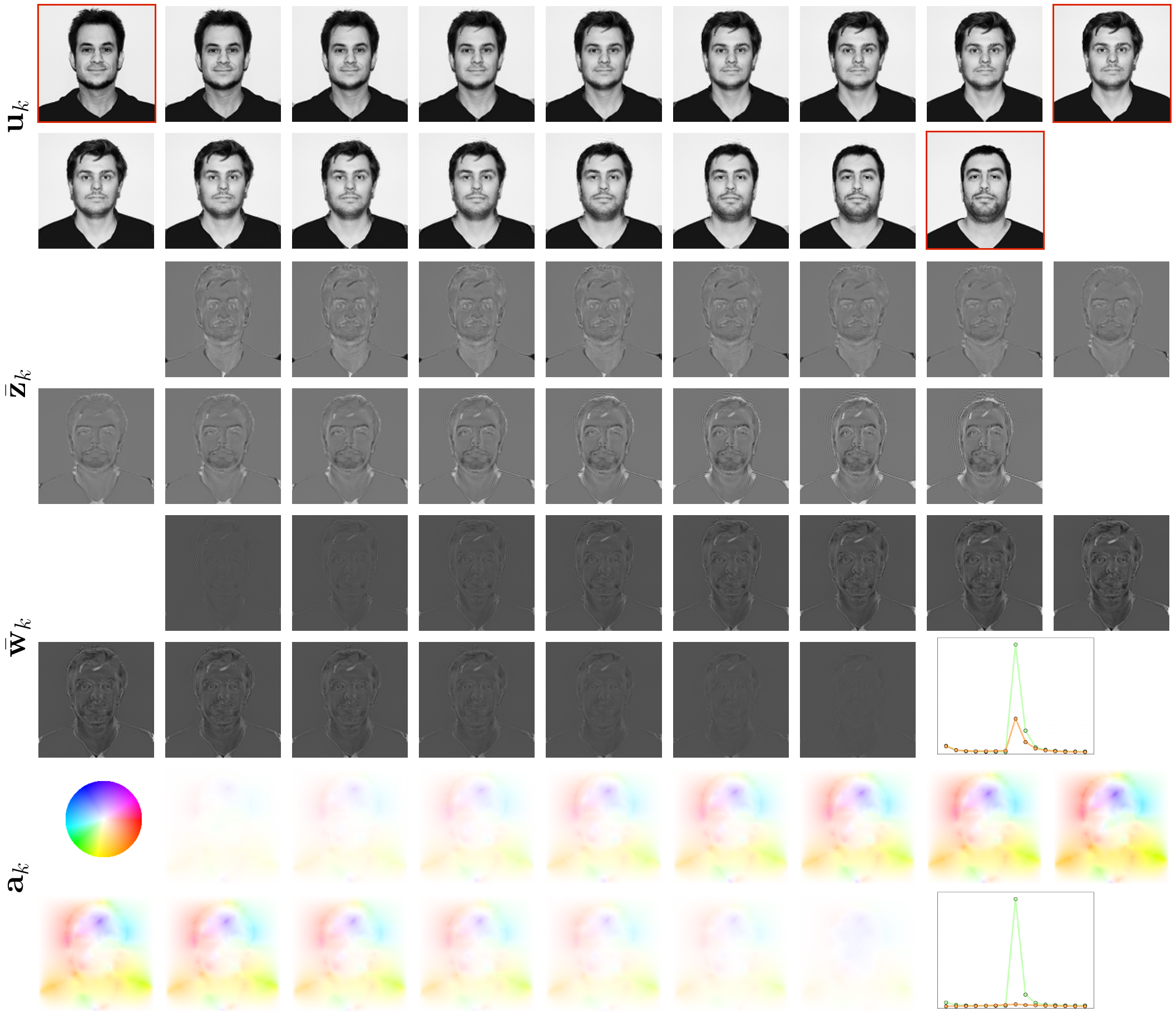}
\caption{Time discrete spline with framed fixed images (first and second row), first order material derivative slack variable $\bar{\mathbf{z}}$ (third and fourth row), second order material derivative with energies comparison (fifth and sixth row) and color-coded acceleration field with energies comparison (seventh and eight row) for values of the parameters $\delta=2\cdot 10^{-2},\,\sigma=2,\,\theta=8\cdot 10^{-4}$. The graphics on the right in row four and six show for the spline (orange) time plots of the $L^2$-norm of the actual second order material derivative $\hat{\mathbf{w}}$ and the dissipation energy density reflecting  motion acceleration \change{$\|\energyDensity_A(\nabla_{\MN}\discreteAcceleration_k)\|_{L^1_{\MN}}$}, respectively. This is compared to the corresponding piecewise geodesic interpolation (green) (not visualized here, cf. Figure \ref{fig:human_faces_comparison}).}
\label{fig:human_faces}
\end{figure*}
The analogous observations hold for Figure~\ref{fig:letters}. 

\change{The impact of the coloring of the key frame images on the geometry of the spline interpolation
and the interplay between the Eulerian flow acceleration and the second order material derivative along motion paths is depicted in Figure~\ref{fig:color_letters}. Therein, we consider block-colored letters as key-frame images. 
In our first example, 
the coloring is consistent with the interpolating flow shown in the black and white spline interpolation in 
Figure~\ref{fig:letters}. Hence, in the corresponding spline interpolation 
color is mainly passively transported along this flow rather than blending it (top row).
In the second example, the red patch in the middle key-frame image is chosen to be located top right instead, further away from the red patches in the extremal key-frame images (bottom row). Hence, 
the coloring is no longer consistent to the flow in the black and white example. 
In fact, the transport in particular of the blue color is now 
strongly reconfigured as seen in the first interval between the 'P' and the 'A'. 
Furthermore, in the second interval between the 'A' and the 'Q'  blending from blue to red and from red to blue occurs.} 
\begin{figure*}[htb]
\includegraphics[width=\textwidth]{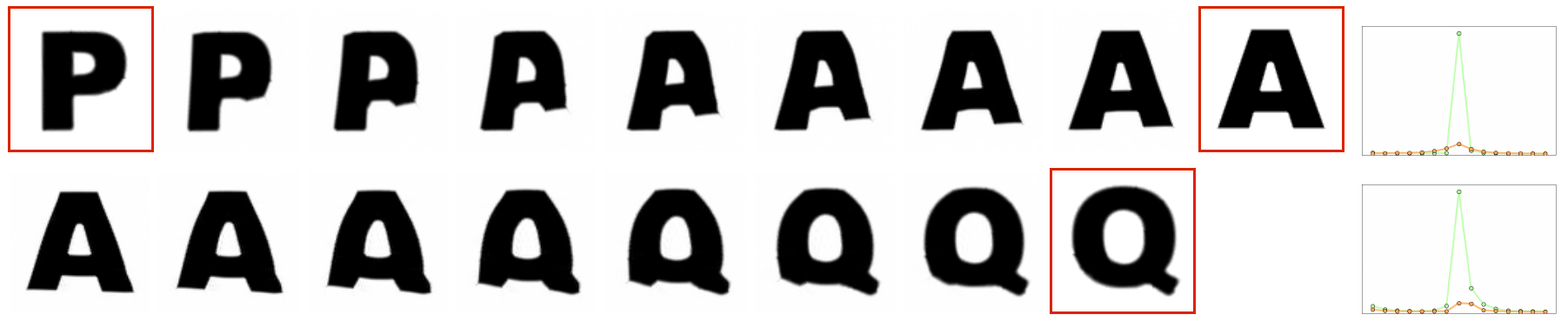}
\caption{Left: Time discrete spline with framed fixed images (top and bottom row). Right: Energy density norm of acceleration flow $\|\energyDensity_A(\nabla_{\MN}\discreteAcceleration_k)\|_{L^1_{\MN}}$ (top), and $L^2$-norm of the actual second order material derivative $\hat{\mathbf{w}}$ (bottom). Parameter values: $\delta=10^{-3}, \sigma=2, \theta=2\cdot 10^{-5}$.}
\label{fig:letters}
\end{figure*}

%%%%%%%%%%%%%%%%%%%%%%%%%%%%%%%%%%%%%%%%%%%%%%%%%%%%%%%%%

\change{
\begin{figure*}[htb]
%\resizebox{\linewidth}{!}{
%\tikzstyle{frame} = [line width=1.8pt, draw=red,inner sep=0.01em]
%\begin{tikzpicture}
%\begin{scope}[scale=1.22]
%\begin{scope}
%\edef\currentCnt{0}
%\foreach \i in  {0,2,4,6,8,10,12,14,16} {
%\ifthenelse{0 = \i \OR 8 = \i \OR 16=\i}{
%\node[frame,anchor=south west] at (0.75*\currentCnt+0.08,0) {\includegraphics[width=0.09\linewidth]{../../images/letters/blue_and_red/option1/spline/u_\i.png}};}
%{
%\node[anchor=south west] at (0.75*\currentCnt,0-0.08) {\includegraphics[width=0.09\linewidth]{../../images/letters/blue_and_red/option1/spline/u_\i.png}};
%}
%\pgfmathparse{\currentCnt+1.9}
%\xdef\currentCnt{\pgfmathresult}
%}
%\end{scope}
%
%\begin{scope}[shift={(0,-1.55)}]
%\edef\currentCnt{0}
%\foreach \i in  {0,2,4,6,8,10,12,14,16} {
%\ifthenelse{0 = \i \OR 8= \i \OR 16=\i}{
%\node[frame,anchor=south west] at (0.75*\currentCnt+0.08,0) {\includegraphics[width=0.09\linewidth]{../../images/letters/blue_and_red/option2/spline_new/u_\i.png}};}
%{
%\node[anchor=south west] at (0.75*\currentCnt,0-0.08) {\includegraphics[width=0.09\linewidth]{../../images/letters/blue_and_red/option2/spline_new/u_\i.png}};
%}
%\pgfmathparse{\currentCnt+1.9}
%\xdef\currentCnt{\pgfmathresult}
%\node at (\currentCnt*0.75-0.65,-0.2) {$\i$};
%}
%\node[rotate=90] at (-0.2,2.25) {$\discreteImage_k$};
%\node[rotate=90] at (-0.2,0.75) {$\discreteImage_k$};
%\node at (-0.2,-0.2) {{$k$}};
%\end{scope}
%\end{scope}
%\end{tikzpicture}
%}
\includegraphics[width=\textwidth]{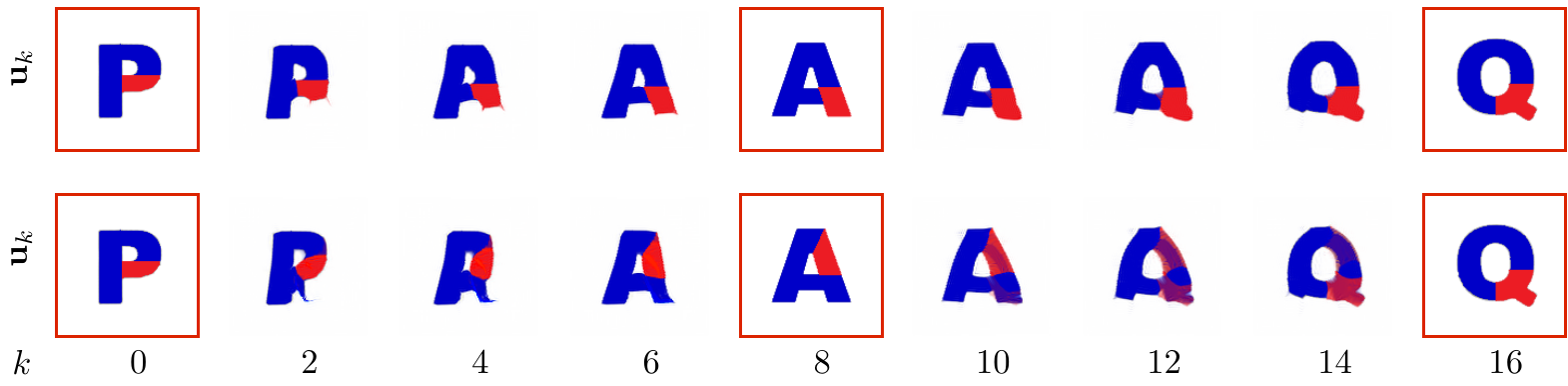}
\caption{Top and bottom rows: Two time discrete splines with key-frames images differing in shape and color. Parameter values: $\delta=8 \cdot 10^{-3}, \sigma=2.5, \theta=2\cdot 10^{-4}$. For visualization purposes, only even-numbered frames $\discreteImage_k$ are depicted.}
\label{fig:color_letters}
\end{figure*}
}

A comparison of splines and piecewise geodesic paths in shown in Figure~\ref{fig:human_faces_comparison}.
One particularly observes that for the faces the shown spline image is thicker and for the letters the spline image shows more round contours than the for the piecewise geodesic counterpart. This is again the non-local impact of key frames beyond those bounding the current interpolation.

\begin{figure*}[htb]
\includegraphics[width=\textwidth]{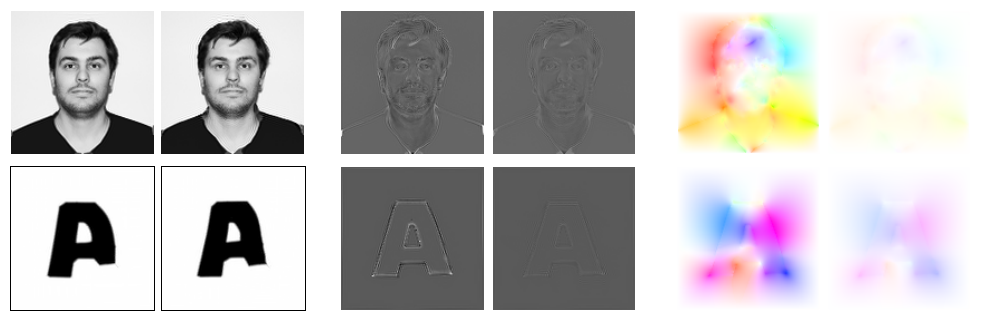}
\caption{Top row: Image $\discreteImage_{11}$ for the human face example, second order material derivative $\hat{\mathbf{w}}_k$ and acceleration field $\discreteAcceleration_k$ for $k=8$, for the time discrete piecewise geodesic (left image of each panel pair) and spline (right image of each panel). The pairs of material derivatives  and the acceleration fields are jointly scaled to reflect the differences in intensities. Bottom row: Same visualization of image $\discreteImage_4$ from the letter example.}
\label{fig:human_faces_comparison}
\end{figure*}
%Fig.~\ref{fig:cells} shows a comparison of the original frames, spline interpolation and piecewise geodesic interpolation for the images extracted from a video made by David Rogers from Vanderbilt University in the 1950s, which shows the interaction between white blood cells and bacteria.
Next, \change{we ask} for a reconstruction of frames given certain frames at selected time stamps extracted from a video.  Here, we compare the resulting spline interpolation and the piecewise geodesic interpolation directly with corresponding frames of the original video as a benchmark for both approaches. Indeed, Figure~\ref{fig:cells} shows this comparison of the original frames, spline interpolation and piecewise geodesic interpolation for the images extracted from a video made by David Rogers from Vanderbilt University in the 1950s\change{\footnote{\url{https://embryology.med.unsw.edu.au/embryology/index.php/Movie_-_Neutrophil_chasing_bacteria}}}, which shows the interaction between white blood cells and bacteria. The spline interpolation clearly shows less blending artifacts in comparison to the piecewise geodesic interpolation.
\begin{figure*}[!h]
\includegraphics[width=\textwidth]{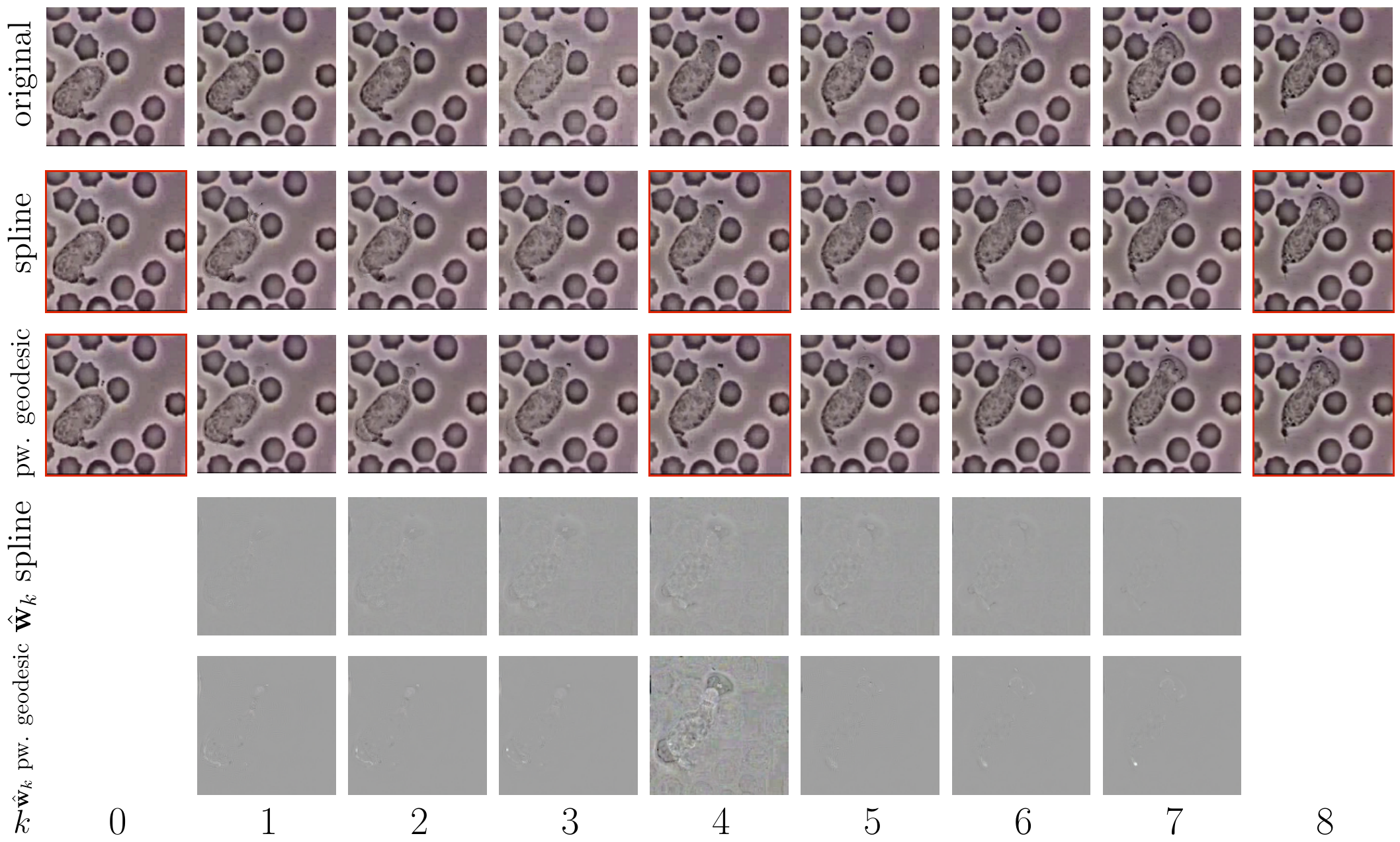}
\caption{First row: The original frames extracted from the video showing white blood cell (neutrophil) chasing a Staphylococcus aureus bacterium. %The original video was made by David Rogers from Vanderbilt University in the 1950s.% 
The images are courtesy of Robert A. Freitas, Institute for Molecular Manufacturing, California, USA. Second row: Time discrete spline with framed fixed images. Third row: Time discrete piecewise geodesic with framed fixed images. Fourth row: Fully discrete second order material derivative for the discrete spline interpolation. \change{Fifth row: Fully discrete second order material derivative for the discrete piecewise geodesic interpolation.} The values of parameters are $\delta=4\cdot 10^{-2}$, $\sigma=2.5$, $\theta=1.6 \cdot 10^{-4}$.}
\label{fig:cells}
\end{figure*}

\change{\section{Conclusion}
In this paper we introduced a novel model for the smooth interpolation of a set of given key frame images based on the 
image metamorphosis model. This is achieved by means of a minimization of the spline energy functional. Thereby, one penalizes acceleration reflected in the underlying transport and Eulerian acceleration defined in terms of the second order change of image intensity along the flow paths. Based on \cite{BeEf14,EfKoPo19a} we proposed a variational time discretization by approximation of the aforementioned quantities by finite differences. We prove consistency of this time discrete model to the time continuous counterpart by the means of Mosco convergence. Space discretization is based on finite differences and approximation of warping operation by cubic B-splines, while for numerical optimization we use the iPALM algorithm. We present a range of numerical experiments where we compare the differences between the piecewise geodesic interpolation and our novel spline interpolation in the space 
image. The quantitative results show that the spline interpolation achieves significant improvement with respect to the smoothness in time, which is indicated in better control and regularity of acceleration quantities in the vicinity of 
the key frames. Furthermore, the influence of the preceding and/or succeeding key frame(s) is visible for the spline interpolation, which naturally relates to the larger support of the spline basis functions for Euclidean cubic splines. 
Still open is an investigation of the differences between our spline model based on the splitting of acceleration quantities and the model based on the  Riemannian second order covariant derivative as relevant acceleration quantity. Furthermore, a challenge is the adaptation of our approach to textures and feature based image representations. 
}

%\begin{table}
%\begin{tabular}{lll}
%\hline\noalign{\smallskip}
%\end{tabular}
%\end{table}


\paragraph{Acknowledgements}
This work was partially supported by the Deutsche Forschungsgemeinschaft (DFG, German Research
Foundation) via project 211504053 - Collaborative Research Center 1060 and project 390685813 - Hausdorff
Center for Mathematics.

\change{The datasets generated during and/or analysed during the current study are available from the corresponding author on reasonable request.}

\bibliographystyle{alpha}     
\bibliography{jmiv_bibliography}

\end{document}